\documentclass[11pt]{article}
\usepackage{xcolor,fullpage}
\usepackage{float}
\usepackage{tikz}
\usepackage{color,amssymb}
\usepackage[font=footnotesize]{caption} 
\usepackage{enumerate}
\usepackage[shortlabels]{enumitem}
\usepackage{authblk}
\usepackage{bm}
\usepackage{amsmath,amsthm}

\RequirePackage{marginnote,hyperref}
\newcommand{\aside}[1]{\marginnote{\scriptsize{#1}}[0cm]}
\newcommand{\aaside}[2]{\marginnote{\scriptsize{#1}}[#2]}
\newcommand\Emph[1]{\emph{#1}\aside{#1}}
\newcommand\EmphE[2]{\emph{#1}\aaside{#1}{#2}}

\usepackage{comment,cleveref}

\newtheorem{lemma}{Lemma}
\newtheorem{theorem}[lemma]{Theorem}
\newtheorem{clm}[lemma]{Claim}
\newtheorem{corollary}[lemma]{Corollary}
\newtheorem{proposition}[lemma]{Proposition}

\newtheorem{question}[lemma]{Question}

\newtheorem*{mainplanarthm}{Main Planar Theorem}
\theoremstyle{definition}
\newtheorem{remark}[lemma]{Remark}
\newtheorem{construction}[lemma]{Construction}

\DeclareMathOperator{\ISRTJ}{ISR-TJ}
\DeclareMathOperator{\ISRTS}{ISR-TS}

\def\C{\mathcal{C}}

\title{A Linear Kernel for Independent Set\\ Reconfiguration in Planar Graphs\footnote{A preliminary version
of this manuscript appeared in STACS 2026.}}

\author{
Nicolas Bousquet\thanks{{CNRS - Université de Montréal CRM - CNRS} and {Univ. Lyon, Université Lyon 1, CNRS, LIRIS UMR CNRS 5205}; \texttt{nicolas.bousquet@cnrs.fr}}
\quad \quad {
Daniel W. Cranston}\thanks{William \& Mary, Department of Mathematics; Williamsburg, VA, USA; \texttt{dcransto@gmail.com}}} 

\usepackage[font=footnotesize]{caption} 
\usetikzlibrary{positioning, quotes}

\definecolor{mypurple}{RGB}{208,134,255}
\definecolor{myblue}{RGB}{10,120,253}
\definecolor{myredred}{RGB}{163,0,0}
\colorlet{myred}{myredred!70}
\definecolor{mygreen}{RGB}{126,198,52}
\definecolor{myorange}{RGB}{244,154,33}

\tikzstyle{uStyle}=[shape = circle, minimum size = 4pt, inner sep = 1pt,
outer sep = 0pt, semithick, draw] %fill=gray!30, 
\tikzstyle{lStyle}=[shape = circle, minimum size = 5pt, inner sep =
0.5pt, outer sep = 0pt, draw=none,fill=none]

\tikzset{every node/.style=uStyle}

\def\offset{.15mm}
\DeclareRobustCommand\greenarrow{
    \begin{tikzpicture}
        \draw[thick, mygreen] (0,0) edge[bend left=30, ->, shorten <= 1.3mm, shorten >= .8mm] (1,0); 
    \end{tikzpicture}
}
\DeclareRobustCommand\dottedround{
    \begin{tikzpicture}[scale=.7, xscale=.5]
       \draw[semithick, rounded corners=.9mm, dashed, gray] (1.5,.85) -- (2.2,.85) -- (2.2,1.15) -- (.8,1.15) -- (.8, .85) -- (1.5,.85);
    \end{tikzpicture}
}

\begin{document}

\maketitle

\begin{abstract}
Fix a positive integer $r$, and a graph $G$ that is $K_{3,r}$-minor-free.
Let $I_s$ and $I_t$ be two independent sets in $G$,
each of size $k$.  We begin with a ``token'' on each vertex of $I_s$ and seek to
move all tokens to $I_t$, by repeated ``token jumping'', removing a single token
from one vertex and placing it on another vertex.
We require that each intermediate arrangement of tokens again specifies an
independent set of size $k$.  
Given $G$, $I_s$, and $I_t$, we ask whether there exists a sequence of token
jumps that transforms $I_s$ into $I_t$.  
When $k$ is part of the input, this problem is known to be PSPACE-complete.
But it was shown by Ito, Kami\'nski, and Ono~\cite{IKO} to be
fixed-parameter tractable.  That is, the problem
can be solved in time $f(k)\cdot Poly(n)$, for some function $f$ and polynomial $Poly(n)$, 
where $n$ denotes the order of $G$.  

Here we strengthen the
upper bound on the running time in terms of $k$ by showing that the problem
has a kernel of size linear in $k$.  More precisely, we transform an arbitrary input problem on a 
$K_{3,r}$-minor-free graph (for some fixed positive integer $r$) into an equivalent problem on a 
($K_{3,r}$-minor-free) graph with order $O(k)$.  This answers positively a question of Bousquet, 
Mouawad, Nishimura, and Siebertz~\cite{BMNS-survey} and improves the
recent quadratic kernel of Cranston, M\"{u}hlenthaler, and Peyrille~\cite{CMP24+}.
For planar graphs, we further strengthen this upper bound to get a kernel of size at most $39k$.
\end{abstract}

\section{Introduction}

The \emph{combinatorial reconfiguration} framework aims at investigating algorithmic and structural aspects of the solution space of an underlying \emph{base problem}. For example, given an instance of some problem $\Pi$ along with two feasible solutions $I_s$ and $I_t$, called the \emph{source} and \emph{target} feasible solutions, our goal is to determine if (and in how few steps) we can transform the source into the target via a sequence of adjacent feasible solutions. Such a sequence is called a \emph{reconfiguration sequence} and every step in the sequence (going from one solution to an adjacent one) is called a \emph{reconfiguration step}. 
Reconfiguration problems arise in various fields such as combinatorial games, motion of robots, random sampling, and enumeration. This framework has been studied extensively for various rules and types of problems in the last twenty years. 
The surveys~\cite{Nishimura17,van2013complexity,BMNS-survey} give a more complete overview of the field. 
In this paper we focus on transformation between independent sets; this study was initiated in~\cite{HD} and motivated by 
planning motion of robots~\cite{HSS-robot}.

Given a simple undirected graph $G$, a subset of $V(G)$ is \emph{independent} if it induces no
edges.  Finding an independent set of maximum cardinality, i.e., the {\sc Independent Set} problem,
is a fundamental problem in algorithmic graph theory; it is known to be NP-hard.  Furthermore, this problem is
not approximable within a factor of $O(n^{1-\varepsilon})$, for any $\varepsilon > 0$, unless P $=$ NP~\cite{Zuckerman07}. 

We view an independent set as a collection of tokens placed on vertices, where no two tokens are adjacent.
We can thus naturally define adjacency relations between independent sets, also called \emph{reconfiguration steps}.
In this paper, we focus on the  {\sc Token Jumping} (TJ) model, introduced by Kami\'{n}ski et al.~\cite{KaminskiMM12}, where a 
single reconfiguration step consists of first removing a token on some vertex $v$ and then immediately adding it back on any 
other vertex $w$, as long as no two tokens become adjacent. The token \Emph{jumps} from vertex $v$ to vertex $w$.
In the {\sc Token Jumping Independent Set Reconfiguration} problem (abbreviated as $\ISRTJ$),
we are given a graph $G$ and two independent sets $I_s$ and $I_t$ of $G$. Our goal is to determine
whether there exists a sequence of reconfiguration steps (called a \Emph{reconfiguration sequence}) that
transforms $I_s$ into $I_t$. 

The $\ISRTJ$ problem has received considerable attention. Hearn and Demaine proved that the problem is PSPACE-complete even restricted to planar graphs of maximum degree $3$~\cite{HD} and Wrochna proved that it remains PSPACE-complete even restricted to graphs of bounded bandwidth~\cite{Wrochna18}. On the positive side, this problem can be decided in polynomial time in certain restricted graph classes, such as line graphs and some of their extensions~\cite{BonsmaKW14,BartierBM24} and even-hole-free graphs 
(for this latter class,~\cite{KaminskiMM12} gave a linear-time algorithm). 
The hardness of the problem motivates studying the problem from a parameterized perspective.

\paragraph*{Parameterized Algorithms.}
A problem $\Pi$ is \Emph{FPT} (fixed-parameter tractable) parameterized by a parameter $x$ if there exists a function $f$ and a polynomial $P$ such that for every instance $\mathcal{I}$ of $\Pi$ of size $n$ for which parameter $x$ has value $k$, the problem can be decided in time $f(k) \cdot Poly(n)$. One can easily prove that the existence of an FPT algorithm is equivalent to the existence of a \Emph{kernel} of size $f(k)$ (for a function $f$), which is an algorithm that provides in polynomial time an \emph{equivalent instance}\footnote{That is an instance which is positive if and only if the original instance is positive.} of size $f(k)$. A kernel is \emph{polynomial} if $f$ is a polynomial function.

\textsc{Independent Set} is known to be $W[1]$-complete~\cite{DowneyF95} but admits FPT algorithms on $K_{r,r}$-free 
graphs and even semi-ladder-free graphs~\cite{FabianskiPST19}.
Its reconfiguration counterpart, $\ISRTJ$ is also W[1]-hard\footnote{Unless $\mathrm{FPT}=\mathrm{W}[1]$, which is a standard assumptions, a W[1]-hard problem admits no FPT algorithm.} parameterized by the size $k$ of the independent sets~\cite{ItoKOSUY14}. One can naturally wonder what is happening when we restrict to structured graph classes. 
In~\cite{IKO}, Ito et al. showed that the $\ISRTJ$ problem is FPT parameterized by $k$ on planar graphs, and 
even on $K_{3,r}$-free graphs (for each fixed value of $r$), i.e., graphs containing no copy of 
$K_{3,r}$ as a subgraph.
This result has been generalized to nowhere dense graphs~\cite{LokshtanovMPRS18} and $K_{r,r}$-free 
graphs~\cite{BMP}. For a thorough history of parameterized aspects of independent set reconfiguration, we refer the reader to~\cite{BMNS-survey}.

Although the parameterized behavior of $\ISRTJ$ is now well understood, its kernelization counterpart remains largely unexplored. 
Even beyond $\ISRTJ$, a deep understanding of kernelization aspects of reconfiguration problems remains elusive 
(see the end of the present section). 
Although it was not mentioned in~\cite{IKO}, their proof can be easily adapted to get a polynomial-size kernel; for instance, 
see~\cite{BMP}.  In a recent survey paper, Bousquet, Mouawad, Nishimura, and Siebertz~\cite{BMNS-survey} 
asked whether $\ISRTJ$ admits a linear kernel for all planar graphs. Cranston, M\"{u}hlenthaler, and Peyrille~\cite{CMP24+} proposed 
a quadratic kernel, 
even for graphs embedded on surfaces.  In this paper, we  answer affirmatively the question of~\cite{BMNS-survey} in a strong sense: 
we prove that $\ISRTJ$ admits a linear kernel when, for any positive integer $r$, we restrict to $K_{3,r}$-minor-free graphs.

\begin{theorem}
For each positive integer $r$, $\ISRTJ$ admits a kernel of
size linear in $k$ on $K_{3,r}$-minor-free graphs. 
\label{main-thm}
\end{theorem}

The case of planar graphs is of particular interest, so there we make the bound more precise.

\begin{theorem}
For planar graphs, $\ISRTJ$ admits a kernel of size at most $39k$.
\label{main-planar-thm}
\end{theorem}

The general approach that we use to prove each of these theorems is the same.  But for planar graphs we work harder to optimize the multiplicative constant.
We note that the kernel we output is actually a subgraph of the initial graph.
This is, as far as we know, the first linear kernel for a reconfiguration problem. As a direct byproduct, we obtain an algorithm running in $2^{O(k)}\cdot Poly(n)$ to decide $\ISRTJ$. As far as we know, it is the first non-trivial single-exponential algorithm for a reconfiguration problem. (At the end of this section, we discuss in more detail kernelization algorithms for other reconfiguration problems.) 

To obtain a kernel that is linear, our argument must be more global than for the kernels in~\cite{IKO,CMP24+}. Their proofs are based on a neighborhood decomposition argument. Let $I_s$ and $I_t$ be the source and target independent sets and let $X:=I_s \cup I_t$. They prove in~\cite{CMP24+} that the number of classes of neighborhoods in $X$ is linear using the neighborhood complexity of planar graphs. It is easy to show that classes with at most $1$ or at least $3$ neighbors in $X$ have linear
size in total. The hardest part of the proof consists in reducing the size of classes with 2 neighbors in $X$. In~\cite{BMP,IKO,CMP24+}, the proofs show that each class can be individually reduced to a linear number of vertices (in $k$). To get a linear kernel, we instead use a more subtle global approach, which is different from previous methods. We prove that we can keep only a constant number of vertices in each class and a linear number (in $k$) of well-chosen vertices in \textit{the union
of the $2$-classes} (the set of all vertices with exactly 2 neighbors in $X$) and still preserve the existence of a solution. As far as we know, this is the first time that such a global 
argument has been used for reconfiguration.

In Sections~\ref{overview:sec} and~\ref{C1C3:sec},
we explain in more detail the key ingredients needed to prove our Main General Theorem. 
And in Section~\ref{C2:sec} we complete its proof.

In the second part of this paper (Section~\ref{sec:planar}), we explain how we can improve the kernel size in the specific case of planar graphs, to get a kernel of size $39k$. This bound is probably still far from optimal, so we tried to compromise between a simple analysis and a small kernel size. As far as we know, no lower bound on kernel sizes has been proposed for reconfiguration. In the $\ISRTJ$ PSPACE-hardness proof of Hearn and Demaine~\cite[Theorem~23]{HD}, the graph obtained in the reduction has size $3k$ and it seems very complicated to compact in polynomial time, so $3k$ appears to be a natural candidate for a lower bound. 
(In fact, their proof is written for $\ISRTS$, but because their reduction uses only \emph{maximum} 
independent sets, it works equally well for $\ISRTJ$.)
Proving a non-trivial lower bound, or improving further the upper bound to a single-digit 
constant, remains a challenging open problem.

The existence of a polynomial kernel for $\ISRTJ$ on graph classes beyond $K_{3,r}$-minor-free graphs remains wide open. The proof of Bousquet, Mary, and Parreau~\cite{BMP} for $K_{r,r}$-free graphs generalizes the method of~\cite{IKO}, but yields a kernel of size $k^{f(r)}$ where $f$ is an exponential function. We thus ask the following question.

\begin{question}
Does $\ISRTJ$ admit a kernel of size $f(r) \cdot Poly(k)$ on graphs of treewidth at most $r$? 
On $K_r$-minor-free graphs? On $K_{r,r}$-free graphs?
\end{question}

\paragraph*{Related work: Kernelization and Reconfiguration.}
\label{related-work-sec}
While the existence (and non-existence) of FPT algorithms for reconfiguration problems has recently been widely studied, almost no polynomial 
kernels have been proposed. Mouawad et al.~\cite{MouawadNRSS17} proved that TJ-Vertex Cover Reconfiguration and TJ-Feedback Vertex Set Reconfiguration both admit quadratic kernels. (Now tokens still move by jumps, as above, but the set of vertices with tokens must be a vertex cover or a feedback vertex set.)
But the existence of linear kernels for these problems is still open. In the 
specific case of vertex cover, we note that this contrasts with the optimization setting, where in recent decades numerous 
classical linear kernels for vertex cover have been discovered; see e.g.~\cite{Abu-KhzamFLS07}. 
Similarly, this is only true for parameter $k$, because for ``distance'', other examples are known, e.g.,
flip distance~\cite{BCK}.

While the existence of a 
linear kernel for independent set is trivial in the optimization setting (by the 4 Color Theorem, every planar graph on at least $4k$ vertices is a \textsf{yes}-instance), the proof 
of \Cref{sec:main} requires significant work. Many meta-kernelization algorithms guarantee the existence of linear kernels on 
planar graphs for optimization problems~\cite{AlberFN04,GuoN07}, e.g., for dominating sets. While it remains open to determine whether 
\textsc{TJ-Dominating Set Reconfiguration} admits a linear kernel, several reduction rules of~\cite{AlberFN04} cannot be adapted 
directly for reconfiguration. 

Important machinery has been developed to prove that problems do not admit polynomial kernels, even if they 
admit FPT algorithms (AND and OR compositions for instance). As far as we know, this framework has never been 
used for reconfiguration problems.  Some of the few results known~\cite{JKKPP} in this context are when the 
parameter is ``distance''.  Finding such a problem, or finding a problem that admits a polynomial kernel for 
the optimization setting but not for its reconfiguration counterpart, remains interesting and open.

Finally, in this paper we focus on Token Jumping. Another model, called Token Sliding (TS), has been studied. 
In the 
\Emph{TS model}, tokens can only move along edges of the graph. Both problems remain PSPACE-complete on planar graphs and on  graphs of bounded bandwidth~\cite{HD,Wrochna18}, but their complexities differ on many graph classes, such as chordal graphs~\cite{BelmonteKLMOS21} and bipartite graphs~\cite{LokshtanovM19}, where the sliding model is harder than its jumping 
counterpart. From a parameterized viewpoint, very little is known for the problem $\ISRTS$. 
While $\ISRTJ$~is known to be FPT even on $K_{r,r}$-free graphs, the parameterized complexity of $\ISRTS$ is open even on graphs 
of bounded treewidth. $\ISRTS$ is known to be FPT on planar graphs~\cite{BBM23} and on graphs with constraints on the 
girth~\cite{BBDLM21,BartierBHMS24}. But the existence of polynomial kernels for $\ISRTS$ remains wide open.
For dominating sets, the token-sliding model again appears much harder than its jumping counterpart since DSR-TS is XL-complete even on bounded treewidth graphs~\cite{BousquetDMMP25,BGS}, while DSR-TJ is FPT on 
nowhere dense classes of graphs~\cite{siebertz}.

\section{Main General Theorem}\label{sec:main}

\subsection{Proof Overview and Key Ideas}
\label{overview:sec}
This section is devoted to proving our Main General Theorem.
Fix a positive integer \Emph{$r$}.
Fix an input graph $G$ with no $K_{3,r}$-minor, along with source and target independent sets, $I_s$ and
$I_t$\aside{$I_s$, $I_t$, $k$}, each of size $k$.  We will show that either
$\ISRTJ(G,I_s,I_t)$ is a 
trivial \textsf{yes}-instance, or else $\ISRTJ(G,I_s,I_t)$ is equivalent to a problem
$\ISRTJ(G',I_s,I_t)$, where $G'$ is a subgraph of $G$ and $|V(G')|=O(k)$.
Let $X:=I_s\cup I_t$ and note that $|X|\le 2k$.  The set $X$ is called the set of \EmphE{key vertices}{-8mm}.
For each $Y\subseteq X$, the $X$-projection\aaside{$X$-projec\-tion, $\C_Y$}{-4mm} \emph{class} {$\C_Y$}
is defined by $\C_Y:=\{v\in V(G)\setminus X\mbox{ such that }N(v)\cap X=Y\}$; the vertices of $Y$ are the
\emph{key vertices of $\C_Y$}.
Let 
$$\C_1:=\bigcup_{\substack{Y\subseteq X\\|Y|\le 1}}\C_Y  \mbox{~~~~and~~~~}
\C_2:=\bigcup_{\substack{Y\subseteq X\\|Y|= 2}}\C_Y  \mbox{~~~~and~~~~}
\C_3:=\bigcup_{\substack{Y\subseteq X\\|Y|\ge 3}}\C_Y.\aside{$\C_1$, $\C_2$, $\C_3$}
$$ 
For each integer $\ell$, we say that the $X$-projection class is an \Emph{$\ell$-class} if $|Y|=\ell$. 
In other words, $\C_1$ and $\C_2$ are the unions respectively of the $1$-classes (and 
$0$-class) and of the $2$-classes, and $\C_3$ consists of all the other classes. 
So $V(G) = X\cup\C_1\cup \C_2\cup C_3$.  Recall that $|X|\le 2k$.  

We will first bound the sizes of $\C_1$ and $\C_3$, in Section~\ref{subsec:C1C3}. 
As already observed in~\cite{IKO}, if $|\C_1|\ge \chi(G) \cdot k$, 
then $\ISRTJ(G,I_s,I_t)$ is a \textsf{yes}-instance. 
However, computing $\chi(G)$ is hard.  So we will work with a parameter $\chi_{up}(G)$ such that
always $\chi(G)\le \chi_{up}(G)$ and we can efficiently construct a $\chi_{up}(G)$-coloring of $G$.
(For example, for planar graphs we let $\chi_{up}(G):=4$ and for general $K_{3,r}$-minor-free graphs,
we let $\chi_{up}(G):=r+2$.)
Thus, we assume $|\C_1|\le \chi_{up}(G) \cdot k$.
Let $N_2(G):=|\{Y\subseteq X:|Y|=2\mbox{ and }\C_Y\ne\emptyset\}|$,\aside{$N_2(G)$} and
let $N_3(G):=|\{Y\subseteq X:|Y|\ge 3\mbox{ and }\C_Y\ne\emptyset\}|$\aaside{$N_3(G)$}{4mm}.
Since $G$ has no $K_{3,r}$-minor, for each $Y$ with $|Y|\ge 3$, we have $|\C_Y|\le r-1$.
So $|\C_3|\le (r-1)N_3(G)$.  Thus, we aim to show that $N_3(G)=O(k)$.  Proving this is easy for
all planar graphs (and, more generally, for all graphs on each fixed surface), as we will show in \Cref{planar-complexity}.
The proof for all $K_{3,r}$-minor-free graphs is more challenging~\cite{JR-complexity, BBFQR-complexity}; but the result 
still holds (see \Cref{general-complexity}).

So the core of the proof consists in showing that we can reduce the size of $\C_2$.  If we could show that $|\C_2|=O(k)$, 
then we could take $G$ as its own kernel.  
We cannot prove this directly, although it is true that $N_2(G)$ has size $O(k)$. 
Yet, unsurprisingly, the size of a class $\C_Y$ with $|Y|=2$ is in general not bounded by a constant.
But can a 2-class be very large, say $\omega(k^2)$?

To motivate our approach in the remainder of the paper, we now sketch a crucial idea.
Suppose that a 2-class $\C_Y$ is very large for some 2-element set $Y$ (large enough to necessarily 
contain an independent set of size $kr$).  If it is impossible ever to move a 
token to $\C_Y$ starting from $I_s$, then we can delete all of $\C_Y$ without changing whether we can 
reconfigure $I_s$ into $I_t$, since no vertices of $\C_Y$ can be used in a reconfiguration sequence. 
The same is true if it is impossible ever to move a token to $\C_Y$ starting from $I_t$.
So we assume that neither of these is impossible.

Since $\C_Y$ is very large, it contains an independent set $I_Y$ of size $kr$.
Since $G$ is $K_{3,r}$-minor-free, each vertex other than the two key vertices of $\C_Y$ has at most $r-1$ neighbors
in $\C_Y$ and, in particular, in $I_Y$.  By assumption, starting from $I_s$ we can eventually
move some token to $\C_Y$.  
In the resulting independent set $I'$, both tokens have moved off the vertices
of $Y$ (moving the first of these tokens off is called \Emph{unlocking} $\C_Y$), 
so each of the $k$ vertices with a token has at most $r-1$ neighbors in $I_Y$.  Thus, $I_Y$ contains at least
$|I_Y|-k(r-1)=k$ vertices that are not adjacent to any vertex in $I'$.  So we can move all tokens on vertices 
in $I'$ to these available vertices of 
$I_Y$ (in arbitrary order).  The same is true starting from $I_t$.  Finally, we can move tokens freely
within $I_Y$, since it is an independent set.  So we can reconfigure $I_s$ to $I_t$.
This argument succeeds whenever $|\C_Y|\ge \chi_{up}(G)kr$, since that guarantees an independent set 
$I_Y$ of size $kr$ (the largest color class in a $\chi_{up}(G)$-coloring of 
$G[\C_Y]$).  Thus, when $|\C_Y| > \chi_{up}(G)kr$, we can delete arbitrary vertices of $\C_Y$.  

Following this approach ensures\footnote{When $G$ is $K_{3,r}$-minor-free, it is straightforward to show that $\chi(G)=O(r)$.  In fact, when $r\ge 6300$ Kostochka and Prince~\cite{KP} showed that $G$ is $(r+2)$-degenerate; thus, $\chi(G)\le r+3$.  See \Cref{KP-lem}.} that $|\C_2'|\le \chi_{up}(G)kr N_2(G)=O(k^2)$.
Below we adapt this idea to ensure that $|\C_2'|\le \max\{O(\chi_{up}(G)kr),$ $O(N_2(G))\}$.
The main new idea is that the independent set $I_Y$ above can be spread over multiple 2-classes.  So 
if we unlock a 2-class containing vertices of $I_Y$, then we use its available vertices to receive tokens
from the \emph{next} 2-class containing vertices of $I_Y$, unlocking that one and proceeding by induction.
Informally, $G'$ is formed from $G$ by deleting some vertices of certain ``big'' 2-classes. 
Formally, we defer constructing $G'$ (and defining ``big'' 2-classes) to \Cref{G'const}.
But once we define $G'$ we can prove the next lemma, which is the core of proving our Main General Theorem.

\begin{lemma}
    \label{key-lem}
    If in $G$ we can (a) start from $I_s$ and unlock some big 2-class and also (b) start from $I_t$ and unlock some big 2-class, then
    $\ISRTJ(G',I_s,I_t)$ is equivalent to $\ISRTJ(G,I_s,I_t)$, and both of them are \textsf{yes}-instances.
\end{lemma}

With \Cref{key-lem} (formalized in \Cref{lem:unlock}), we prove our Main General Theorem as follows.

\begin{proof}[Proof of the Main General Theorem] 
    We show $\ISRTJ(G',I_s,I_t)$ and $\ISRTJ(G,I_s,I_t)$ are equivalent.  Since $G'\subseteq G$, 
    if the latter is a \textsf{no}-instance, then so is the former.  So assume instead that $\ISRTJ(G,I_s,I_t)$
    is a \textsf{yes}-instance.  First, suppose that it is impossible, starting from $I_s$, to ever unlock a  big 2-class.  
    So all big 2-classes will always remain locked, and 
    it is impossible to ever move a token to a vertex of a big 2-class.  Thus, since every vertex of
    $V(G)\setminus V(G')$ is in a big 2-class, every reconfiguration sequence in $G$, starting from $I_s$
    is also valid in $G'$.  This proves the desired result.  The argument is identical if it is impossible
    to unlock a big 2-class, starting from $I_t$.  Thus, we assume instead that, starting from both $I_s$
    and $I_t$, it is possible to unlock some big 2-class.  Now we are done by \Cref{key-lem}.
\end{proof}

\subsection{Dealing with \texorpdfstring{$\C_1$}{C1} and \texorpdfstring{$\C_3$}{C3}}\label{subsec:C1C3}
\label{C1C3:sec}
In this subsection, we determine a function $f$ such that if $G$ is $K_{3,r}$-minor-free, then we can assume that $|\C_1|+|\C_3|\le f(r)k$.
And when $G$ is planar, we can improve our bound on $f$.  Bounding $|\C_1|$ is easy, both in the planar case and in the more general case.  But bounding $|\C_3|$ is more work.  For this we use the observation (\Cref{clm:C3}) that $|\C_3|\le (r-1)N_3(G)$; recall here that $N_3(G)$ denotes the number of sets $Y\subseteq X$ with $|Y|\ge 3$ and $\C_Y\ne\emptyset$. To bound $N_3(G)$ in the general (non-planar) case, we use a powerful result (\Cref{nbhd-complexity-lem}) from~\cite{BBFQR-complexity}.

It is straightforward to prove that if $G$ is $K_{3,r}$-minor-free, then $\chi(G) = O(r)$.  But determining the right multiplicative (and additive) constant is more work.  This is done by the following lemma, which is sharp.

\begin{lemma}[Kostochka--Prince~\cite{KP}]
\label{KP-lem}
Fix $r\ge 6300$.  If $G$ is an $n$-vertex graph with $n\ge r+3$ and $G$
has no $K_{3,r}$-minor, then $2|E(G)| \le (r+3)(n-2)+2$.  Thus, $G$ is $(r+2)$-degenerate
and $\chi(G)\le r+3$.
\end{lemma}

\begin{lemma}\label{clm:C1}
If $|\C_1|\ge \chi_{up}(G)k$, then $\ISRTJ(G,I_s,I_t)$ is a \textsf{yes}-instance.
In particular, since $G$ is $K_{3,r}$-minor-free, this is true whenever $r\ge 6300$ and $|\C_1|\ge k(r+3)$.
\label{clm1}
\end{lemma}

\begin{proof}
The second statement follows from the first by \Cref{KP-lem}; thus, we prove the first.

Assume $|\C_1|\ge \chi_{up}(G)k$.  By definition, $G$ is $\chi_{up}(G)$-colorable, and we can
compute such a coloring efficiently.
By Pigeonhole, $\C_1$ contains an independent set $I_m$ (for middle) of size 
$\chi_{up}(G)k/\chi_{up}(G)=k$.  
Starting with tokens on $I_s$, for each $v\in I_s$ with a neighbor $w_v\in
I_m$, move the token on $v$ to some such $w_v$.  Now move all remaining tokens
(in an arbitrary order) to the unoccupied vertices of $I_m$.  By symmetry, we
can also move all tokens from $I_t$ to $I_m$.  Thus, we have a \textsf{yes}-instance of
$\ISRTJ(G,I_s,I_t)$, as claimed.
\end{proof}

Henceforth we assume $|\C_1|<\chi_{up}(G)k$.  Recall from above that $N_3(G)$ 
denotes the number of sets $Y\subseteq X$ with $|Y|\ge 3$ and $\C_Y\ne\emptyset$.  
The next lemma follows directly from the fact that $G$ is $K_{3,r}$-minor-free.

\begin{lemma}\label{clm:C3}
$|\C_3|\le (r-1)N_3(G)$.
\label{clm2}
\end{lemma}
\begin{proof}
Suppose the lemma is false.  By Pigeonhole, there exists $Y\subseteq X$ with $|\C_Y|\ge \lceil |\C_3|/N_3(G)\rceil>\mbox{$(r-1)N_3(G)/N_3(G)$}$.  
That is, $|\C_Y|\ge r$.  But now $G$ contains the subgraph $K_{3,r}$ with the vertices in the part of size 3 in $Y$ and those in 
the part of size $r$ in $\C_Y$.  This contradicts that $G$ is $K_{3,r}$-minor-free.
\end{proof}

When $G$ is planar, we can use Euler's formula to improve the bound above.
\begin{lemma}
\label{planar-complexity}
If $G$ is planar, then $N_2(G)\le 3|X| \le 6k$ and $N_3(G)\le 2|X| \le 4k$.
\end{lemma}
\begin{proof}
We draw a plane graph $G_X$ with vertex set $X$ where each set $Y\subseteq X$ with $|Y|=2$ and 
$C_Y\ne\emptyset$ corresponds to an edge of $G_X$.  Think of restricting $G$ to $X$ and one vertex 
$v_Y$ in $C_Y$ for each such $Y$ with $|Y|= 2$ (deleting any edges among vertices of $X$).  
For each $v_Y$, we now contract exactly one of its two incident edges.  Note that for the resulting plane 
graph $G_X$ its number of edges is precisely $N_2(G)$, the number of $2$-classes of $G$.
By Euler's Formula, $G_X$ has at most $3|X|-6$ edges (when $|X|\ge 3$; the desired weaker inequality
$N_2(G)\le 3|X|$ is immediate in the small cases), so $N_2(G)\le 3|X|\le 6k$. 

For every nonempty projection class $C_Y$ with $|Y|\ge 3$, we choose one representative $x_C$ 
(recall that $x_C \notin X$).  We denote by $p$ the number of such classes, i.e., $p:=N_3(X)$.
We now construct the bipartite graph with parts $X$ and 
$\{x_Y: |Y|\ge 3, \C_Y\ne\emptyset \}$; 
that is, we consider the subgraph induced by these vertices, but with all edges removed that have 
both endpoints inside $X$ or have both endpoints inside 
$\{x_Y: |Y|\ge 3, \C_Y\ne\emptyset \}$. 
The resulting graph is indeed bipartite and planar.
By Euler's formula, the number of edges of a bipartite planar graph is at most twice its number of vertices. 
So we have $3p < 2(p+|X|)\le 2(p+2k)$, which implies that $p < 2|X|\le 4k$, as claimed.
\end{proof}

The \Emph{neighborhood complexity} of a graph class $\mathcal{G}$ is the smallest function $f$, if it exists, such that for all $G\in \mathcal{G}$, nonempty $A\subseteq V(G)$, and nonnegative integers $s$, we have the bound $|\{N^s[v]\cap A: v\in V(G)\}| \le f(s)|A|$.  Here $N^s[v]$ is the set of vertices at distance at most $s$ from $v$.  We need the following result.
\begin{lemma}[{\cite[Theorem 18]{BBFQR-complexity}}]
\label{nbhd-complexity-lem}
For all positive integers $s,t$ with $t\ge 4$, for every $K_t$-minor-free graph $G$, for every set $A$ of vertices of $G$,
$$|\{N^s[v]\cap A: v\in V(G)\}|
\le 4^t(t-3)t^{2(t-1)}(s+1)^{3(t-1)}|A|.$$
\end{lemma}
Since we are interested only in neighborhoods (that is, distance 1), we let $s:=1$.  Since $G$ is $K_{3,r}$-minor-free, it is also $K_{3+r}$-minor-free.  So we let $t:=r+3$.  Finally, we let $A:=X$ and recall that $|A|=|X|\le 2k$.
\begin{corollary}
\label{general-complexity}
$N_2(G)+N_3(G)\le 4^{r+3}r(r+3)^{2r+4}2^{3r+6}(2k)\le 2^{5r+13}(r+3)^{2r+5}k$.
\end{corollary}
Combining the results in this subsection, we get the following.
\begin{lemma}
$|\C_1|+|\C_3|\le k(\max\{r,6300\}+3+(r-1)(2^{5r+13}(r+3)^{2r+5}))$.  If $G$ is planar, then $|\C_1|+|\C_3|\le {12}k$.
\end{lemma}
\begin{proof}
We start with the first statement.
If $G$ is $K_{3,r}$-minor-free, with $r\le 6300$, then also $G$ is $K_{3,6300}$-minor-free. So the bound on $|\C_1|$ follows from \Cref{clm:C1}, and the bound on $|\C_3|$ follows from \Cref{clm:C3} and \Cref{general-complexity}.
Summing these bounds gives the first statement.

Now we prove the second statement.  By \Cref{clm:C1}, we assume that $|\C_1|\le 4k$.
By \Cref{clm2} (with $r:=3$) and \Cref{planar-complexity} we get that $|\C_3|\le (3-1)(4k)=8k$.
Summing these bounds gives the second statement.
\end{proof}
\subsection{Bounding the size of \texorpdfstring{$\C_2$}{C2}}
\label{C2:sec}
The rest of the proof consists in showing that the following lemma holds.

\begin{lemma}\label{lem:boundC2}
We can find in polynomial time an equivalent instance, formed by possibly deleting some vertices of $\C_2$, 
	to get a subset $\C_2'$ for which $|\C_2'|\le \chi_{up}(G)(N_2(G)(4r-1)+k)$.
\end{lemma}

First we construct our graph $G'$, by (possibly) deleting some vertices of $\C_2$, and we show that 
$|\C_2'|\le \chi_{up}(G)(N_2(G)(4r-1)+k)$.  This is fairly straightforward.  Afterwards, we show that
$\ISRTJ(G',I_s,I_t)$ is equivalent to $\ISRTJ(G,I_s,I_t)$.  We sketched this latter step above. So all that
remains for us is to handle the harder case: when it is possible starting from $I_s$ to move a token to some
vertex of $\C_2\setminus \C_2'$, and this is also possible starting from $I_t$.

\begin{construction}
    \label{G'const}
    To form $G'$ from $G$, we do the following. We call the final independent set $I$.
    \begin{enumerate}
	    \item[(1)] If $|\C_2|\le \chi_{up}(G)(N_2(G)(3r-2)+k)$, then do nothing; that is, $\C_2':=\C_2$.
        \item[(2)] Otherwise, by Pigeonhole pick $I_0\subseteq \C_2$ with $I_0$ an independent set and $|I_0|=N_2(G)(3r-2)+k$.
	\item[(3)] A 2-class $\C_Y$ is \Emph{big} if $|\C_Y|\ge \chi_{up}(G)(2r-1)+1$; otherwise, $\C_Y$ is \EmphE{small}{4mm}. 
        \item[(4)] For each small 2-class, do nothing. 
        \item[(5)] For each big 2-class $\C_Y$, do the following.
        \begin{itemize}
            \item[(a)] If $\C_Y$ has at least $3r-1$ vertices of $I_0$, then 
                delete all vertices of $\C_Y\setminus I_0$.
            \item[(b)] If $\C_Y$ has at most $3r-2$ vertices of $I_0$, then: \\
            Keep in $\C_Y$ an arbitrary independent set of size $2r$ and remove all other vertices of $\C_Y$. \\
            Remove all the vertices of $I_0 \cap \C_Y$ from $I_0$.
        \end{itemize}
    \end{enumerate}
\end{construction}

Note that the independent set $I$ is also being returned as part of the output of Construction~\ref{G'const}.
The independent sets $I$ and $I_0$ from Construction~\ref{G'const} play a key role in the rest
of Section~\ref{C2:sec}.

\begin{proposition}
	We have $|\C_2'|\le \chi_{up}(G)(N_2(G)(4r-1)+k)$.
\end{proposition}
\begin{proof}
If $|\C_2|\le \chi_{up}(G)(N_2(G)(4r-1)+k)$, then we are done, trivially.  So assume not.  Now we define $I_0$
and delete vertices of big 2-classes as in \Cref{G'const}.  If a 2-class $\C_Y$ is either small or intersects 
$I_0$ in at most $3r-2$ vertices, then in $\C_2'$, we keep at most $\chi_{up}(G)(2r-1)$ vertices of $\C_Y$.
Thus, the total number of vertices in these classes (restricted to $G'$) is at most $N_2(G)\chi_{up}(G)(2r-1)$.

If a 2-class $\C_Y$ is big and intersects $I_0$ in at least $3r-1$ vertices, then in $G'$ we keep in 
$\C_Y$ only its vertices in $I_0$.  Thus, the total number of vertices in these classes (restricted to $G'$) 
is at most $|I_0|\le N_2(G)(3r-2)+k$.  So the total size of $\C_2'$ is at most 
$N_2(G)\chi_{up}(G)(2r-1)+N_2(G)(3r-2)+k\le \chi_{up}(G)(N_2(G)(4r-1)+k)$.
\end{proof}

A \Emph{helpful independent set} is any subset of $I$ of size $k$.
(Step (5b) in Construction~\ref{G'const} removes at most $N_2(G)(3r-2)$ vertices from $I_0$.
Thus $|I|\ge k$.)
As $I$ is independent, the following holds. 

\begin{remark}\label{rem:allgraal}
    Any helpful independent set can be transformed into any other.
\end{remark}

For a $2$-class $\C_Y$, we call the 2 vertices in $Y$ the \Emph{key vertices} of $\C_Y$.
To \EmphE{unlock}{7mm} a big 2-class $\C_Y$ is to move tokens to reach an independent set $I'$
such that $|I'\cap Y|\le 1$.  After we unlock a class $\C_Y$, we can move the 
single token in $N(\C_Y)$, if it exists, onto $\C_Y$ and then move additional tokens onto $\C_Y$ (provided
that $\C_Y$ contains a large enough independent set).

\begin{lemma}\label{lem:minor2vertices}
    Fix $x,x'\in X$ and let $C:=C_{\{x,x'\}}$. Every component of $G[V \setminus (C \cup \{x,x'\})]$ is adjacent to at most $r-1$ vertices of $C$.
\end{lemma}
\begin{proof}
If not, then $G$ contains a $K_{3,r}$-minor, where one side consists of the vertices $x,x'$, and the 
component $A$ with $r$ neighbors in $C$, and the other side consists of $r$ vertices of $N(A) \cap C$.
\end{proof}
\begin{remark}
Having defined $G'$ and bounded its size, all that remains is to prove that this new instance $\ISRTJ(G',I_s,I_t)$ is equivalent to the original $\ISRTJ(G,I_s,I_t)$.  This equivalence is precisely
the assertion of  \Cref{key-lem}, and it 
follows immediately from \Cref{lem:unlock}.
\end{remark}

\begin{lemma}\label{lem:unlock}
If in $G$ we can from $I_s$ (resp.~$I_t$) unlock a big $2$-class, then in $G'$ we can from $I_s$ (resp.~$I_t$) 
reach a helpful independent set.  That is, there exists a transformation from $I_s$ (resp.~$I_t$) into a 
helpful independent set in $G$ that only uses vertices of $G'$.
\end{lemma}
\begin{proof}
Assume there exists an independent set \Emph{$J_0'$} that can be reached from $I_s$ such that $|J_0'\cap Y_C|\le 1$, where $Y_C$ is the key pair 
for some big 2-class $C$. Among all transformations from $I_s$ to any such\footnote{Here we are not specifying a particular set $J_0'$, but we are instead minimizing over all transformations from $I_s$ to independent sets reachable from $I_s$ that unlock a big 2-class.} set $J_0'$, we take a transformation 
\Emph{$\mathcal{R}$} of minimum length.
The case when $\mathcal{R}$ has length 0 (that is $J_0'=I_s$) is easier, so we handle it briefly at the end.
For now we assume that $\mathcal{R}$ has positive length.

We claim that: (i) the final step of $\mathcal{R}$ consists in moving a token on a key vertex $x$ of class $C$ to some vertex 
\Emph{$z$} and, 
(ii) each jump of the transformation except the last one consists of moving a token (from its current vertex) to an adjacent vertex.
Point (i) follows from the minimality of the transformation. At some step in $\mathcal{R}$, we move a token away from a key vertex
of some big 2-class.  If $\mathcal{R}$ continues with further steps, then we can omit these steps, contradicting the minimality of $\mathcal{R}$.
Point (ii) holds because if, prior to the final step of $\mathcal{R}$, we moved a token from a vertex $v$ to a vertex $w$, 
with $w$ not adjacent to $v$, then we should have instead moved the token on the key vertex $x$ to $w$; this gives a shorter transformation, again contradicting the minimality of $\mathcal{R}$.
Note also that no big 2-class is unlocked before the final step of $\mathcal{R}$, and that the final 
destination is neither a vertex nor a key vertex of any big 2-class (otherwise, we again contradict the 
minimality of $\mathcal{R}$).

We form \Emph{$G''$} from $G$ by deleting all vertices in big classes and all key vertices for big classes.
Now (ii) above implies that all vertices of $G''$ that have gained or lost a token during $\mathcal{R}$ {(excluding the key vertex incident to a big class, which loses a token on the last move of $\mathcal{R}$)} must 
be in the same component of $G''$; otherwise we can simply omit from $\mathcal{R}$ all moves in components 
of $G''$ other than the component where we move our token on our final move, which unlocks $C$.

We claim that each big class $C'$ has at most $2(r-1)$ vertices with neighbors in $J_0'\setminus I_s$, as follows.
Let \Emph{$z$} be the vertex where we move a token on the final step of $\mathcal{R}$ (unlocking $C$).
Note that all vertices of $J'_0 \setminus (I_s\cup \{z\})$ belong to the same component of $G''$; this 
follows from the previous paragraph, since each of these vertices received a token during $\mathcal{R}$ (excluding its last step). 
By \Cref{lem:minor2vertices}, each big class $C'$ has at most $r-1$ vertices with neighbors in $J'_0 \setminus 
(I_s\cup\{z\})$.  
And also by \Cref{lem:minor2vertices}, each big 2-class $C'$ has at most $r-1$ neighbors of $z$.  Thus, the claim holds.
    
By the definition of $J'_0$, there is a big 2-class $C$ that has been unlocked; that is $J'_0$ only contains a token on at most
one of the two key vertices of $C$.  Regardless of whether or not $C$ contained (in $G$) at least $3r-1$ vertices of $I_0$,
we know $C$ contains in $G'$ an independent set of size at least $2r$.
From above, at most $2r-2$ of these vertices have neighbors in $J'_0\setminus I_s$.  So at least $2r-(2r-2)=2$ of them
have no such neighbor.  We denote by $\{a,b\}$ an independent set of size 2 in $C \setminus N(J'_0 \setminus I_s)$.

By \Cref{lem:minor2vertices}, the set $\{a,b\}$ has at most $r-1$ neighbors in each other big class $C'$. 
Actually, let $x,y$ denote the key vertices for the big 2-class $C$ that contains $a,b$.  
For each big class $C'$, with $C'\ne C$, one of $x$ and $y$ is not a key vertex for $C'$; by symmetry, we assume $x$ is not.  
If $|N(\{a,b\})\cap C'|\ge r$, then we get a $K_{3,r}$-minor
by contracting $\{a,b\}$ onto $x$ (with $x$ and the two key vertices of $C'$ as the part of size 3).  
Thus, $|N(\{a,b\})\cap C'| \le r-1$.
So for each projection class $C'$ containing at least $3r-1$ vertices of $I_0$, we know that
$|(C' \cap I_0) \setminus N(\{a,b\} \cup (J'_0 \setminus X))| \ge (3r-1) - (r-1)-(r-1)-(r-1)=2$. 
We denote by $I'$ the set $I \setminus N(\{a,b\})$. 
Note in particular that by construction $I'$ contains at least $2$ vertices in each projection class 
containing at least $3r-1$ vertices of the initial reservoir $I_0$.
        
We will perform the following sequence of moves. First, the token on the remaining key vertex of $C$ is moved onto $a$. We can indeed  move it to $a$ since the other vertices of $J_0'$ are not adjacent to $\{a,b\}$.
Let $C_1,\ldots,C_\ell$ be the big classes intersecting $I$ in $G'$ (different from $C$ if $C$ also intersects $I$). Let us denote by $a_i,b_i$ two vertices of $I'$ in $C_i$ and $x_i,y_i$ the two key vertices of $C_i$. We set $b_0:=b$. For increasing $i$, we perform the following moves: we move (if it exists) the token on $x_i$ to $b_{i-1}$ and the token on $y_i$ to $a_i$. Since $I'$ is non-adjacent to $J_0'$ and $I'$ is an independent set, we can easily check that this defines a sequence of independent sets. 

At the end of this transformation, we get an independent set $J_\ell$ where every big class containing vertices of $I$ has been unlocked. Moreover, by construction, vertices of $J_\ell \setminus I$ have at most $3r-3$ neighbors in each big class.
When constructing $G'$ from $G$, we deleted vertices of $I_0$ only in step (5b).  So the number of vertices of 
$I_0$ we deleted is at most $N_2(G)(3r-2)$.
Thus, the total number of vertices in $I_0$ (in $G$) that are unavailable in $G'$ to receive tokens from 
$J_\ell$ is at most $N_2(G)(3r-2)$.  So the number of vertices available to receive tokens from $J_\ell$ 
is at least $|I_0|-N_2(G)(3r-2)=k$.
Hence, in $G'$ all the tokens of $J_\ell$ can one by one be moved to $I$, as desired.

Now we remark briefly on the case that $\mathcal{R}$ has length 0; that is, some big 2-class $C$ is already unlocked.
In this case, we just begin immediately moving tokens to $a,b$ in $C$.  Now each big 2-class $C'$ has at most $r-1$ vertices
in $N(\{a,b\})$, but we do not need to worry about neighbors of $z$ in the component of $G''$ where other moves occurred.
So the analysis above still holds.
\end{proof}

\section{Improved Kernel for Planar Graphs}\label{sec:planar}

\subsection{Proof Outline}

Now we provide an algorithm that outputs a smaller kernel in the specific case of planar graphs. 
More precisely, the goal of \Cref{sec:planar} is to prove the following result.

\begin{theorem}\label{planar-thm}
On planar graphs $\ISRTJ$ admits a kernel of size $39k$.
\end{theorem}

The general idea of the proof is similar to that of \Cref{lem:unlock}. We prove that if $G$ is large enough, 
then $G$ contains an independent set $I$ of size at least $k$ with the following property: if we can unlock one of 
the ``big classes'' from $I_s$ (or $I_t$), then we can transform $I_s$ (or $I_t$) into a size $k$ subset of $I$. 
To obtain a smaller kernel, we need two main ingredients. First, we give a more subtle way to define this independent set 
$I$. This allows us to find such an independent set $I$ in planar graphs much smaller than required by \Cref{lem:unlock}. 
Second, we prove that, as in \Cref{lem:unlock}, if we can unlock one class that is big enough, then we can transform 
$I_s$ into $I$. The difference from \Cref{lem:unlock} is that our new notion of ``big enough'' is actually much smaller;
but this savings comes at the cost of slightly more involved analysis.

We now explain in more detail how we find this set $I$ and describe some of its properties. Let $X:=I_s \cup I_t$.
Let $Y$ be a subset of vertices. The \Emph{$X$-neigh\-borhood of $Y$} is $\cup_{y \in Y} N(y) \cap X$.
One of the key ideas in the proof is the concept of (weakly) greedy independent sets. Let $I$ be an independent subset of 
$V \setminus X$ consisting of vertices in $2$-classes.  

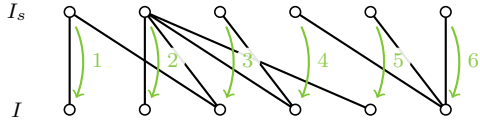
\begin{figure}
\begin{minipage}[t]{.45\textwidth}
% figure 1
\begin{figure}[H]
    \centering
\begin{tikzpicture}[yscale=1.3, every edge quotes/.style = {outer sep=1pt, right, font=\footnotesize, draw=none, fill=white, fill opacity = .9}, xscale=.95]
    \foreach \i in {1,...,6}
    {
        \draw (\i,1) node (x\i) {} (\i,0) node (y\i) {};  % draw nodes
        \draw (x\i) ++ (-1*\offset,0) node[lStyle] (xp\i) {} (y\i) ++ (1.5*\offset,0) node[lStyle] (yp\i) {}; 
        % to push tails of arrows left toward nodes and heads of arrows right away from nodes
    }

    \draw[thick] (y1) -- (x1) -- (y3) -- (x2) -- (y2) (y5) -- (x2) -- (y4) -- (x3) (x6) -- (y6) -- (x5) (y6) -- (x4); % all edges

    \draw (.3,1) node[lStyle] {\footnotesize{$I_s$}} (.3,0) node[lStyle] {\footnotesize{$I$}};

    \foreach \i in {1,...,6}
        \draw[thick, mygreen] (xp\i) edge[bend left=30, ->, shorten <= 1.3mm, shorten >= .8mm, "\i"] (yp\i); % draw all arrows

    \draw[white] (0,1.52) --++ (0.1,0);  % to help fine tune spacing of the figure
    \draw[white] (0,-.45) --++ (0.1,0);
\end{tikzpicture}
\captionsetup{width=.65\textwidth}
    \caption{An $I_s$-greedy independent set, together with the transformation \greenarrow
    }
    \label{fig:greedy}
\end{figure}
\end{minipage}
\begin{minipage}[t]{.6\textwidth}
% figure 2
\begin{figure}[H]
    \centering
\begin{tikzpicture}[yscale=1.3, every edge quotes/.style = {outer sep=1pt, right, font=\footnotesize, draw=none, fill=white, fill opacity = .9}, xscale=.95]
    \clip (-.5,-.5) rectangle + (7.5,2);    
    \foreach \i in {1,...,6}
    {
        \draw (\i,1) node (x\i) {} (\i,0) node (y\i) {};  % draw nodes
        \draw (x\i) ++ (2,0) ++ (-1*\offset,0) node[lStyle] (xp\i) {} (y\i) ++ (2,0) ++ (1.5*\offset,0) node[lStyle] (yp\i) {}; 
        % to push tails of arrows left toward nodes and heads of arrows right away from nodes
    }

    \draw[thick] (x1) -- (y1) -- (x2) -- (y2) -- (x1) -- (y4) -- (x3) -- (y6) -- (x4) (y3) -- (x2) -- (y5) -- (x5) (x6) -- (y6);

    \draw (.3,1) node[lStyle] {\footnotesize{$I_s$}} (.3,0) node[lStyle] {\footnotesize{$I$}};

    \foreach \i in {1,...,4}
        \draw[thick, mygreen] (xp\i) edge[bend left=30, ->, shorten <= 1.3mm, shorten >= .8mm, "\i"] (yp\i); % draw all arrows

    \draw[semithick, rounded corners=2mm, dashed, gray] (1.5,.85) -- (2.2,.85) -- (2.2,1.15) -- (.8,1.15) -- (.8, .85) -- (1.5,.85);
    \begin{scope}[yshift=-1cm]
    \draw[semithick, rounded corners=2mm, dashed, gray] (1.5,.85) -- (2.2,.85) -- (2.2,1.15) -- (.8,1.15) -- (.8, .85) -- (1.5,.85);

    \draw (x1) ++ (0,.35) node[lStyle] {\footnotesize{$j_1$}} (x2) ++ (0,.35) node[lStyle] {\footnotesize{$j_2$}};
    \draw (y1) ++ (0,-.325) node[lStyle] {\footnotesize{$i_1$}} (y2) ++ (0,-.325) node[lStyle] {\footnotesize{$i_2$}};
    \end{scope}
\end{tikzpicture}
\captionsetup{width=.65\textwidth}
    \caption{A weakly $I_s$-greedy independent set. {\dottedround} are the activation pairs and \greenarrow is the transformation.
    }
    \label{fig:weaklygreedy}
\end{figure}
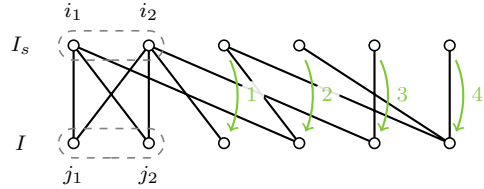
\end{minipage}
\end{figure}

We say that $I$ is \Emph{$I_s$-greedy} (resp.~$I_t$-greedy) if there is a greedy algorithm that moves tokens (one-by-one) from the vertices of $I_s$ (resp.~$I_t$) onto vertices of $I$, while keeping an independent set all throughout the transformation; see Figure~\ref{fig:greedy}. To rephrase, this means that, at each step of the transformation, we can find a vertex of $I_s \setminus I$ that can be replaced by a vertex of $I \setminus I_s$. 
Again equivalently, but in a more structural way, there is an ordering $j_1,\ldots,j_k$ of $I_s$ and $i_1,\ldots,i_k$ of $I$ such that, for every $t \le k$, vertices $i_1,\ldots,i_t,j_{t+1},\ldots,j_k$ form an independent set. 
Note that if $I_1$ and $I_2$ are both independent sets of size $k$, then $I_1$ is $I_2$-greedy if and 
only if $I_2$ is $I_1$-greedy.  That is, being greedy is symmetric.

We will also need the following weakening of $I_s$-greedy independent sets.
The set $I$ is \EmphE{weakly $I_s$-greedy}{-4mm} if there exist vertices $j_1,j_2 \in I_s$ and $i_1,i_2 \in I$ such that $(I_s \setminus \{j_1,j_2\}) \cup \{ i_1,i_2 \}$ is an independent set, call it $I'$, and $I$ is $I'$-greedy; see Figure~\ref{fig:weaklygreedy}. 
In other words, there are orderings $j_1,\ldots,j_k$ of $I_s$ and $i_1,\ldots,i_k$ of $I$ such that, for every $t\in\{2,\ldots,k\}$, the vertex subset $\{i_1,\ldots,i_t,j_{t+1},\ldots,j_k\}$ is independent. 
Note that an $I_s$-greedy independent set is indeed weakly $I_s$-greedy.  If $I$ is weakly $I_s$-greedy, then the pair $\{j_1,j_2\}$ (resp.~$\{i_1,i_2\}$) is called the \Emph{$I_s$-ac\-tiva\-tion pair} (resp.~$I$-activation pair). These pairs ``activate'' the 
transformation in the sense that, if $\{j_1,j_2\}$ has been replaced by $\{i_1,i_2 \}$, then we can greedily finish the 
transformation from $I_s$ into $I$.
(In the rest of the proof, $I$ might have size larger than $k$, since we simply want to transform $I_s$ into a subset of $I$; 
but imagining that these set sizes are equal keeps all the hardness of the problem.)

Assume that $G$ contains an independent set $I$ of size at least $k$ that is weakly $I_s$-greedy and weakly $I_t$-greedy. If 
$I$ is $I_s$-greedy and $I_t$-greedy, then we can transform $I_s$ into $I_t$, passing through $I$.
But if $I$ is only weakly $I_s$-greedy, then nothing ensures that we can transform $I_s$ into $I$. Nevertheless, by definition, if we can replace the activation vertices of $I_s$ with the activation vertices of $I$ (and similarly for $I_t$), then we can transform 
$I_s$ into $I$. But (i) there might not exist any transformation between $I_s$ and $I$ and, (ii) if there is a transformation, nothing guarantees that some such transformation satisfies this condition.

To overcome point (ii), we exhibit certain special weakly greedy independent sets (called weakly clean independent sets) in 
Section~\ref{sec:clean}. And we also prove that if $G$ is large enough, then either there is no transformation from $I_s$ into $I$ or 
there is a transformation that replaces the activation pair of $I_s$ by the activation pair of $I$ without moving the other tokens. Moreover, in the latter case, this transformation
only uses a constant number of vertices in each $2$-class, as well as (possibly) vertices in classes that 
are not $2$-classes.  (We optimize this constant knowing that the graph is planar.)
Finally, we prove that every planar graph that is large enough has an independent set $I$ that is both weakly $I_s$-greedy and 
weakly $I_t$-greedy; this completes the proof.

\paragraph{Organization.}
In Section~\ref{sec:firstobs}, we start with a few observations. 
In Section~\ref{sec:clean}, we define clean independent sets, and 
prove that if $G$ is large enough, then it contains an independent set $I$ that is clean both for $I_s$ and for $I_t$. 
In Section~\ref{sec:combining} we combine all these arguments to get the desired smaller kernel, assuming the truth of a key lemma
about transformations of $I_s$ into $I$.  Finally, in Section~\ref{sec:weaklyclean}, we prove this key lemma.

\subsection{First Observations}\label{sec:firstobs}
When constructing a sequence to reconfigure $I_s$ to $I_t$, we would prefer to be able to move tokens onto vertices of distinct 
2-classes independently of each other.  But this may be impossible, because of edges between some of these vertices.
So this first subsection is about how we can allow ourselves this desired freedom.
Throughout this section, we extensively use the following simple remark; it is due to the fact that all 
vertices belonging to a given $2$-class are adjacent to the same two key vertices.

\begin{remark}\label{rem:pathorcycle}
If $G$ is a planar graph, then the vertices of each $2$-class $C$ induce either a cycle or a disjoint union of paths.
\end{remark}

\begin{proof}
    If some vertex $v$ in $C$ has at least 3 neighbors in $C$, say $w_1,w_2,w_3$, then $G$ contains as a subgraph $K_{3,3}$, with
    $\{w_1,w_2,w_3\}$ in one part and with $v$ and the key vertices for $C$ in the other.  This $K_{3,3}$ contradicts that $G$ is
    planar.  And if $G[C]$ contains a cycle $H$, then $H$ separates the key vertices, so no other vertex is adjacent to both 
    key vertices.  Thus, $H$~spans~$C$.
\end{proof}

In particular, every $2$-class $C$ contains an independent set of size 
$\lfloor |C|/2 \rfloor$.  Moreover, if we consider a strict subset of $C$, that is a set $D \subsetneq C$,  
then $D$ induces a disjoint union of paths; so $D$ admits an independent set of size at least 
$\lceil |D|/2 \rceil$. The subset $D$ is typically formed from $C$ by deleting $2$ 
vertices. We usually remove vertices to guarantee that $2$-classes are anticomplete to each other (disjoint 
vertex subsets $A$ and $B$ are \Emph{anticomplete} to each other if no edge of $G$ has one endpoint in $A$ 
and the other in $B$). Namely, we often use the following remark.

\begin{remark}\label{rem:CC'}
If $C,C'$ are distinct $2$-classes, then the following 2 statements hold.
\begin{enumerate}[(a)]
    \item  The set $N(C) \cap C'$ has size at most $2$. Moreover, if $N(C)\cap C'$ has size $2$ then its two elements must be either
    (i) consecutive vertices on a path or cycle of $G[C']$; or  
    (ii) two endpoints of a single path of $G[C']$; or    
    (iii) two endpoints of disjoint paths of $G[C']$.    
    \item If $x,y$ are the key vertices of $C$, then, for every 
    $z \in X$ distinct from $x,y$, the union of the classes incident to $z$ has at most two neighbors in $C$.
        (Otherwise, $G$ has a $K_{3,3}$-minor.)
\end{enumerate}
\end{remark}

In the case of planar graphs, we can actually strengthen \Cref{rem:CC'}(a) to apply to more than two classes, and we prove the following version.

\begin{lemma}\label{lem:deletion2perclass}
Let $G$ be a planar graph. Let $C_1,\ldots,C_r$ be $2$-classes.  For every $i<r$, there exist $C_i' \subseteq  C_i$ with $|C_i'| \ge |C_i|-2$ such that $C_1',\ldots,C_{r-1}',C_r$ are pairwise anticomplete.
\end{lemma}
\begin{proof}
Consider a plane drawing of the subgraph of $G$ induced by $\cup_{i \le r} C_i \cup X$ such that the outer 
face of the drawing contains either the 2 key vertices of $C_r$ or 2 vertices contained in $C_r$, 
possibly both.  For every class $C_i$, let $G_i$ denote the subgraph induced by $C_i$ and its 2 key vertices.
A class $C_i$ is \Emph{nested} in $C_j$ if all the vertices of $C_i$ lie within a face of $G_j$ distinct 
from its outer face. 
    Two classes $C_i$ and $C_j$ are \Emph{incomparable} if the vertices of $C_a$ are on the outer face of 
$G_b$ whenever $\{a,b\}=\{i,j\}$. For every $2$-class $C$, the vertices of $C$ plus its two key vertices form 
a $K_{2,|C|}$ (possibly with extra edges among $C$), so every pair of classes is either nested or incomparable.

Now, for every $i < r$, let $B_i$ denote the at most two vertices of $C_i$
on the boundary of $G_i$, and let $C_i':=C_i\setminus B_i$.
We claim that $C_i'$ is anticomplete to $C_j'$ for every $j \ne i$ (with $C_r':=C_r$). If the two classes are incomparable, then the only vertices of $C_i$ that can be adjacent to $C_j$ are the vertices of $B_i$ and $B_j$, but the vertices of at least one of these sets are excluded from $C'_i\cup C'_j$ (and the vertices of both are if $r\notin \{i,j\}$). Otherwise, up to symmetry, $C_i$ is nested in $C_j$. 
By the definition of the plane drawing, $i \ne r$ and the vertices of the outer face of $C_i$ are excluded 
from $\C'_i$.
\end{proof}

Later we will need a slight variation of Lemma~\ref{lem:deletion2perclass}. (Its proof is nearly identical to that above.)

\begin{lemma}\label{lem:deletion2perclass_upd}
Let $G$ be a planar graph, and let $C_1,C_2$ be two $2$-classes. If $C$ is the vertex set of a connected subgraph of 
$G[V \setminus (C_1 \cup C_2 \cup N((C_1 \cup C_2) \cap X))]$, then the following 2 statements hold.
\begin{enumerate}[(a)]
\item There exist $C_1',C_2'$ such that $|C_i'| \ge |C_i|-2$ for every $i\in\{1,2\}$ and $C_1',C_2',C$ are pairwise anticomplete.
\item If $C_1$ and $C_2$ are anticomplete, then there exists $C'\subseteq C$ such that $|C'|\ge |C|-4$ and $C',C_1, C_2$ are
pairwise anticomplete. 
\end{enumerate}
\end{lemma}

\subsection{Clean Independent Sets}\label{sec:clean}

We define two types of independent sets; these are similar to the helpful independent sets of Section~\ref{sec:main}.
A weakly $I_s$-greedy independent set $I$ is \Emph{$3$-clean} \emph{for $I_s$} if it satisfies 
both of the 2 conditions below; see Figure~\ref{fig:3clean}. 
\begin{enumerate}
\item Some $2$-class $C$ contains at least three vertices of $I$; and
\item The vertices of the activation pair\footnote{There might exist several orders, and several activation pairs. 
In this case, we select one pair (and corresponding order) that satisfies the condition and call this \emph{the} activation pair.} of $I$ are in $C$.
\end{enumerate}
A weakly $I_s$-greedy independent set $I$ is \Emph{$(2,2)$-clean} \emph{for $I_s$} if it satisfies both of the 2 conditions below; see Figure~\ref{fig:3cleanB}. 
\begin{enumerate}
\item Some vertex $x$ is a key vertex for two $2$-classes $C,C'$ that both contain two vertices of $I$; and
\item The vertices of the activation pair of $I$ consist of two selected vertices of $C \cap I$.
\end{enumerate}

\begin{figure}[!h]
\begin{minipage}[t]{.425\textwidth}
    \centering
\begin{tikzpicture}[yscale=1.3, every edge quotes/.style = {outer sep=1pt, right, font=\footnotesize, draw=none, fill=white, fill opacity = .9}, scale=.85, xscale=.90]
    \foreach \i in {1,...,7}
    {
        \draw (\i,1) node (x\i) {} (\i,0) node (y\i) {};  % draw nodes
    }

    \draw[thick] (x1) -- (y1) -- (x2) -- (y2) -- (x1) -- (y3) -- (x2) (x3) -- (y4) -- (x4) -- (y5) -- (x3) (x5) -- (y6) -- (x6) -- (y7) -- (x7);

    \draw (.3,1) node[lStyle] {\footnotesize{$I_s$}} (.3,0) node[lStyle] {\footnotesize{$I$}};

    \draw[semithick, rounded corners=1.5mm, dashed, gray] (1.5,.85) -- (2.2,.85) -- (2.2,1.15) -- (.8,1.15) -- (.8, .85) -- (1.5,.85);
    \begin{scope}[yshift=-1cm]
    \draw[semithick, rounded corners=1.5mm, dashed, gray] (1.5,.85) -- (2.2,.85) -- (2.2,1.15) -- (.8,1.15) -- (.8, .85) -- (1.5,.85);
    \end{scope}

    \begin{scope}[yshift=-1cm]
    \draw[semithick, rounded corners=2.75mm, gray] (1.5,.75) -- (3.3,.75) -- (3.3,1.25) -- (.7,1.25) -- (.7, .75) -- (1.5,.75);
    \end{scope}

    \begin{scope}[yshift=-1cm, xshift=3cm]
    \draw[semithick, rounded corners=2.75mm, gray] (1.5,.75) -- (2.3,.75) -- (2.3,1.25) -- (.7,1.25) -- (.7, .75) -- (1.5,.75);
    \end{scope}

    \draw[semithick, gray] (y6) circle (3.25mm and 2.5mm) (y7) circle (3.25mm and 2.5mm);
    \draw (5.5,-.4) node[lStyle, shape=rectangle] {$\underbrace{~~~~~~~~~~~~~~~~~~~~~~~~~~~~~~}$};
    \draw (5.5,-.6) node[lStyle, shape=rectangle] {\footnotesize{other 2-classes}};
    \draw (2,-.5) node[lStyle, shape=rectangle] {\footnotesize{$C$}};
\end{tikzpicture}
\captionsetup{width=.725\textwidth}
    \caption{The independent set $I$ is 3-clean for $I_s$. {\dottedround} are the activation pairs.
    }
    \label{fig:3clean}
\end{minipage}
%\hfill
~~~~~~
\begin{minipage}[t]{.45\textwidth}
    \centering
\begin{tikzpicture}[yscale=1.3, every edge quotes/.style = {outer sep=1pt, right, font=\footnotesize, draw=none, fill=white, fill opacity = .9}, scale=.85, xscale=.95]
  \foreach \i in {0,...,7}
    {
        \draw (\i,1) node (x\i) {} (\i,0) node (y\i) {};  % draw nodes
    }

    \draw[thick] (x1) -- (y1)  (x2) -- (y2) -- (x1) -- (y3) -- (x2) (x3) -- (y4) -- (x4) -- (y5) -- (x3) (x5) -- (y6) -- (x6) -- (y7) -- (x7)  (x1) -- (y0) -- (x0) -- (y1);

    \draw (-.7,1) node[lStyle] {\footnotesize{$I_s$}} (-.7,0) node[lStyle] {\footnotesize{$I$}};

    \draw[semithick, rounded corners=1.5mm, dashed, gray] (0.5,.85) -- (1.2,.85) -- (1.2,1.15) -- (-.2,1.15) -- (-.2, .85) -- (0.5,.85);
    \begin{scope}[yshift=-1cm]
    \draw[semithick, rounded corners=1.5mm, dashed, gray] (0.5,.85) -- (1.2,.85) -- (1.2,1.15) -- (-.2,1.15) -- (-.2, .85) -- (0.5,.85);
    \end{scope}

    \begin{scope}[yshift=-1cm]
    \draw[semithick, rounded corners=2.75mm, gray] (0.5,.75) -- (1.3,.75) -- (1.3,1.25) -- (-.3,1.25) -- (-.3, .75) -- (0.5,.75);
    \draw[semithick, rounded corners=2.75mm, gray] (2.5,.75) -- (3.3,.75) -- (3.3,1.25) -- (1.7,1.25) -- (1.7, .75) -- (2.5,.75);
    \end{scope}

    \begin{scope}[yshift=-1cm, xshift=3cm]
    \draw[semithick, rounded corners=2.75mm, gray] (1.5,.75) -- (2.3,.75) -- (2.3,1.25) -- (.7,1.25) -- (.7, .75) -- (1.5,.75);
    \end{scope}

    \draw[semithick, gray] (y6) circle (3.25mm and 2.5mm) (y7) circle (3.25mm and 2.5mm);
    \draw (5.5,-.4) node[lStyle, shape=rectangle] {$\underbrace{~~~~~~~~~~~~~~~~~~~~~~~~~~~~~~}$};
    \draw (5.5,-.6) node[lStyle, shape=rectangle] {\footnotesize{other 2-classes}};
    \draw (1,1.35) node[lStyle, shape=rectangle] {\footnotesize{$x$}};
    \draw (.5,-.5) node[lStyle, shape=rectangle] {\footnotesize{$C$}};
    \draw (2.5,-.5) node[lStyle, shape=rectangle] {\footnotesize{$C'$}};
\end{tikzpicture}
\captionsetup{width=.75\textwidth}
    \caption{The independent set $I$ is (2,2)-clean for $I_s$. {\dottedround} are the activation pairs.
    }
    \label{fig:3cleanB}
\end{minipage}
\end{figure}

The set $I$ is \Emph{weakly clean} \emph{for $I_s$} if it is either $3$-clean or $(2,2)$-clean.
Note that if $I$ is weakly clean for $I_s$, then the $I_s$-activation pair {contains} $N(C) \cap I_s$.  The key vertices of $C$ (and $C'$ for $(2,2)$-clean independent sets) might not be in $I_s$. But in that case, we can easily transform $I_s$ into $I$, as we show in the
next lemma.

\begin{lemma}\label{lem:hardcase}
Let $I$ be a $3$-clean (resp.~$(2,2)$-clean) independent set for $I_s$. If at least one of the two vertices of $N(C) \cap X$ 
(resp.~the three vertices of $N(C \cup C') \cap X$) is not in $I_s$, then we can transform $I_s$ into $I$ using only vertices of $I_s \cup I$. 
\end{lemma}
\begin{proof}
Since $I$ is weakly $I_s$-greedy, we simply need to prove that we can move the tokens from the activation pair of $I_s$ to 
the activation pair of $I$. Indeed, if we can do this, then afterward we can complete the transformation, from $I_s$ into $I$, 
greedily.
    
First suppose that $I$ is $3$-clean, as in Figure~\ref{fig:3clean}. By definition, the activation pair $\{i_1,i_2\}$ of $I$ is in $C$ and the activation pair 
$\{j_1,j_2\}$ of $I_s$ contains the vertex $j_1$ of $N(C) \cap I_s$ (if it exists).  We can replace\footnote{By \emph{replace $a$ with $b$} we mean to move the token from vertex $a$ to vertex $b$; this replaces $a$ with $b$ in the independent set defined by the vertices
currently with a token.} $j_1$ with $i_1$ and replace $j_2$ with $i_2$
 (since $N(C) \cap I_s$ has at most one vertex, namely $j_1$) and the conclusion follows.

Now instead assume that $I$ is $(2,2)$-clean, as in Figure~\ref{fig:3cleanB}.  If $|N(C) \cap I_s|\le 1$, 
then we can perform the same moves and the conclusion follows similarly. So we assume that the activation 
pair of $I_s$ is $\{x,j_2\}$, the two vertices of $N(C) \cap X$.  By assumption not all the vertices of 
$N(C \cup C') \cap X$ are in $I_s$. So the second key vertex of $C'$ is not in $I_s$. Thus, we can replace 
$x$ with a vertex $y'$ of $I \cap C'$, replace $j_2$ with $i_1$, and finally replace $y$ with $i_2$. 
This completes the proof.
\end{proof}

From now on, by Lemma~\ref{lem:hardcase}, we assume that class $C$ is \Emph{locked}, and in the
$(2,2)$ case class $C'$ is also locked; that is, $|N(C)\cap I_s|=2$ and $|N(C')\cap I_s|=2$, the latter only
in the $(2,2)$ case. In particular, the activation pair of $I_s$ is $N(C) \cap I_s$.
If $I$ is $3$-clean (resp.~$(2,2)$-clean), then the vertex of $I \cap C$ (resp.~the two vertices of $I \cap C'$) that is not in the 
$I$-activation pair (if several such vertices exist, then we choose one of them arbitrarily) is called the \EmphE{auxiliary activation vertex}{-4mm} (resp.~vertices) of $I$. Moreover, if $I$ is $(2,2)$-clean, then the key vertex of $C'$ that is not a key vertex of $C$ is called the \emph{auxiliary activation vertex of $I_s$}.

An independent set is \Emph{clean} \emph{for $I_s$} if it is $I_s$-greedy, $3$-clean for $I_s$, or 
$(2,2)$-clean for $I_s$.  
Recall that if $I$ is $I_s$-greedy, then we can transform $I$ into $I_s$. We now explain informally 
how we will make use of an independent set that is $3$-clean or $(2,2)$-clean for $I_s$.  We will argue that 
if we can unlock a large enough class $D$, then (if we consider a good transformation) we can replace an 
$X$-vertex adjacent to $C$ (resp.~replace $x$, when $I$ is $(2,2)$-clean) with a vertex in the class $D$. 
This fact, together with the fact that $C \cap I$ (resp.~$(C \cup C') \cap I$) is large enough will allow 
us to find a transformation from $I_s$ into $I$. 
We formalize this intuition with the following simple example.

\begin{lemma}\label{lem:atmost8}
Let $G$ be a $K_{3,3}$-free graph, and let an independent set $I$ be clean for $I_s$. 
If $|V(G) \setminus N[I_s]|\ge 3$, then we can transform $I_s$ into $I$ using at most one vertex that is not in $I_s \cup I$.  
\end{lemma}
\begin{proof}
The conclusion follows if the independent set $I$ is $I_s$-greedy. So we instead assume that $I$ is weakly $I_s$-greedy. 
Thus, if we can replace the activation pair of $I_s$ by the activation pair of $I$, then we can complete the transformation greedily. 
The rest of the proof consists in showing that we can do this. Let $\{i_1,i_2\}$ be the $I$-activation pair and $i_3$ be an auxiliary 
$I$-activation vertex (and $i_4$ be the other if $I$ is $(2,2)$-clean). We denote by $C$ the class containing $i_1,i_2$ (and by 
$C'$ the class containing $i_3$ if $I$ is $(2,2)$-clean). Since $G$ is $K_{3,3}$-free, there is a non-edge 
between some vertex $a\in V(G) \setminus N[I_s]$ and some vertex $i_b$ of $\{i_1,i_2,i_3\}$ such that $a$ is 
both unoccupied and nonadjacent to every occupied source vertex.

If $I$ is $3$-clean, then we let $\{j_1,j_2\}$ be the activation pair of $I_s$ (these are also the key vertices of $C$).
We replace $j_1$ with $a$, replace $j_2$ with $i_b$, replace $a$ with some vertex $i_c$ in $\{i_1, i_2 \} \setminus i_b$, 
and finally replace $i_b$ with $\{i_1,i_2\} \setminus i_c$ if $i_b=i_3$.
    
If $I$ is $(2,2)$-clean, then we let $\{x,j_1\}$ be the $I_s$-activation pair, and let $j_2$ be the vertex of $N(C') \cap I_s$ 
distinct from $x$. If $i_b$ is in the $I$-activation pair, then we can conclude as above for $3$-clean independent sets. So we assume 
that $i_b=i_3$.  We replace $x$ with $a$, replace $j_2$ with $i_3$, replace $a$ with $i_4$, replace $j_1$ with $i_1$, 
replace $i_3$ with $i_2$, and replace $i_4$ with $j_2$.  
\end{proof}

To conclude this section, we prove that if $G$ is large enough, then the graph contains an independent set that is clean for both 
$I_s$ and $I_t$.
Namely, the following holds.

\begin{lemma}\label{lem:condition}
    Let $G$ be a planar graph and $I_s,I_t$ be two independent sets of size $k$. If the number of vertices in $2$-classes is at 
    least $21k$, then there exists an independent set $I$ of size at most $2k$ that is both clean for $I_s$ and clean for $I_t$.
\end{lemma}
\begin{proof}
Note that if an independent set $I$ is weakly $I_s$-greedy (resp.~$I_t$-greedy), then $I$ remains so when we add vertices to $I$.  
So it suffices to find an independent set $I'$ of size $k$ that is clean for $I_s$, find another $I''$ for $I_t$, and take their 
union, as long as this union is also independent.  To ensure this union is indeed independent, we first find a large 
independent set $I_0$ and choose the clean independent sets $I'$ and $I''$ from within $I_0$.

Let $X:=I_s \cup I_t$.
By \Cref{planar-complexity}, the number of 2-classes in $G$ is at most $3|X| \le 6k$.  For each $2$-class $C$, we call $C$ a 
\Emph{good} \emph{class for $I_s$} (resp.~\emph{for $I_t$}) if $C$ has at most one neighbor in $I_s$. 
A 2-class that is not good is called \Emph{bad}. 

By \Cref{lem:deletion2perclass}, after removing at most $2$ vertices per class, we assume that all the 
$2$-classes are anticomplete to each other.   The total number of vertices that we remove is at most 
$2N_2(X)\le 6|X|\le 12k$.  Moreover, by \Cref{rem:pathorcycle}, the remaining vertices of each 2-class 
induce a disjoint union of paths, so each contains an independent set with at least half its vertices.  We 
denote by \Emph{$I_0$} the union of these independent sets; 
note that $I_0$ is also independent and $|I_0|\ge (21k-12k)/2=4.5k$.
As we mentioned above, we choose from among vertices of $I_0$ an independent set of size $k$ that is clean 
for $I_s$, and we do the same for $I_t$; the union of these two sets is the desired independent set $I$, with 
size at most $2k$.  So below it suffices only
to construct the independent subset of $I_0$ that is clean for $I_s$.
    
If the number of vertices of $I_0$ appearing in good classes for $I_s$ is at least $k$, then we take an 
arbitrary set of $k$ of them; this set is $I_s$-greedy, so we are done.
Thus, we assume instead that the number of vertices of $I_0$ in good classes is at most $k-1$.
Hence, the number of vertices of $I_0$ that are in bad classes is at least $4.5k-(k-1)=3.5k+1$.

If more than $0.5k$ bad classes each have at least $2$ vertices in $I_0$, then by Pigeonhole two such bad 
$2$-classes $C,C'$ share a key vertex of $I_s$.  To form an independent set that is $(2,2)$-clean for $I_s$, 
we take the union of the vertices of $I_0$ in $C,C'$, and then continue picking a bad class (with at least 
two vertices of $I_0$) and adding all of its vertices, until we reach a set of size $k$.  To see that this 
set is $(2,2)$-clean, we start the ordering with the vertices of $C\cup C'$.

So we assume instead that all but at most $0.5k$ bad $2$-classes contain at most $1$ vertex of $I_0$; we call them 
\Emph{small classes}.  We denote by \EmphE{$\ell$}{-4mm} the number of non-small bad classes; thus, $\ell \le 0.5k$.
By the proof of \Cref{planar-complexity}, at most $3k$ classes are bad; so the number of small bad classes is at most $3k-\ell$.
Thus, the number of vertices of $I_0$ in small bad classes is at most $3k-\ell$.  So the number of vertices of $I_0$ in non-small 
bad classes is at least $3.5k+1-(3k-\ell)=
0.5k+\ell+1\ge 2\ell+1$, since $\ell\le 0.5k$.  Thus, by Pigeonhole, some bad $2$-class $C$ has at least $3$ vertices of $I_0$.
To form our independent set $I$ that is 3-clean for $I_s$, we take all of the vertices of $I_0$ in $C$ and add vertices of $I_0$ from bad 
2-classes that are not small, up to a set of size $k$ (or when no remaining bad class has two vertices of $I_0$).
    
To order $I$, we begin with the vertices in 
the bad 2-class $C$ with $|C\cap I_0|\ge 3$, and continue to other selected non-small bad 2-classes.
always making all vertices in a bad 2-class successive in the order.  If we reach a set of size $k$, then this 
independent set $I$ is weakly $I_s$-greedy, and we are done; so we assume we do not reach a set of size $k$.  
    
However, we do reach a set $I$ of size at least $0.5k+\ell+1$. So we need to add to $I$ at most 
$0.5k-\ell-1$ vertices.  After moving in the tokens from all key vertices needed to unlock $I$, these 
$0.5k+\ell+1$ vertices of $I$ have at least $0.5k+\ell+1-2\ell=0.5k-\ell+1$ vertices with no token.  
Thus we can move to them one token from each of exactly 
$k-|I|$ small bad classes containing a vertex of $I_0$.  Afterwards, all of these small bad classes are 
unlocked, so the remaining tokens can move to these vertices.  Thus, the resulting set is weakly $I_s$-greedy.
\end{proof}

\begin{lemma}
\label{cor:condition}
In fact, if there exist $q$ $2$-classes that each have a key vertex outside of $I_s$, and the union of these 
classes has size at least $2q+2k$, then the construction of Lemma~\ref{lem:condition} can choose a size-$k$ 
subset $J$ in the union of these classes that is $I$-greedy.
\end{lemma}
\begin{proof}
We apply Lemma~\ref{lem:deletion2perclass} to these classes.  We delete at most 2 vertices per classes to 
get a subset of vertices that induces a disjoint union of paths.  Thus, we have an independent set of
size at least $((2q+2k)-2q)/2 = k$.  Since each of these classes has a key vertex outside of $I_s$, we can
greedily move the vertices from $I_s$ onto this set.
\end{proof}

\subsection{Size of the Kernel}\label{sec:combining}
In the present section, we finish the proof of our Main Planar Theorem, assuming the truth of \Cref{lem:reduc-equiv}, below.  And 
in the next (and final) section, we prove the truth of \Cref{lem:reduc-equiv}.
A vertex is \emph{important} if it is a key vertex for either (a) at least one $2$-class of size at least $7$ or (b) at least two
$2$-classes of size at least $5$.

\begin{lemma}
\label{lem:reduc-equiv}
Let $G$ be a planar graph and $I$ be an independent set that is weakly clean for $I_s$.  We can transform $I_s$ into an independent
subset of $I$ whenever we can unlock from $I_s$ either (a) a $2$-class of size at least $7$ or (b) a $2$-class of size at least $5$
with a key vertex adjacent to a second $2$-class of size at least $5$.  
    Moreover, such a transformation still exists in every subgraph $G'$ formed from $G$ by deleting vertices (outside of $I$)
in $2$-classes of size at least $5$ such that every vertex important in $G$ remains important in $G'$.
\end{lemma}

The proof of \Cref{lem:reduc-equiv} is a bit technical, which is why we defer it; but it is conceptually 
straightforward.  We denote by $G'$ the graph formed from $G$ by deleting all the vertices allowed by the 
``moreover'' statement.  We explicitly construct in $G'$ a transformation from $I_s$ to $I$.  
In (a) we show that $I$ is $3$-clean for $I_s$ and in (b) 
we show that $I$ is $(2,2)$-clean for $I_s$ (in both cases, w.r.t.~$G'$).

\begin{mainplanarthm}
For planar graphs, $\ISRTJ$ admits a kernel of size at most $39k$.  Furthermore, we can construct this kernel in polynomial time.
\end{mainplanarthm}

\begin{proof}
By \Cref{clm:C1} for planar graphs, if $|\mathcal{C}_1|\ge 4k$, then $(G,I_s,I_t)$ is a \textsf{yes}-instance.
This is easy to check in polynomial time; so we henceforth assume that $|\mathcal{C}_1|\le 4k$.
By \Cref{clm:C3}, we have $|\mathcal{C}_3| \le 8k$.  Since $|I_s\cup I_t|\le 2k$, it suffices to construct an equivalent
instance $(G',I_s,I_t)$, where $G'$ is a subgraph of $G$, with $|\mathcal{C}_2| \le 25k$. If we can do this, then $G'$ 
has order at most $(4+25+8+2)k=39k$. 

So we assume that the number of vertices in $2$-classes (in $G$) is at least $25k$. 
By \Cref{lem:condition}, we can find an independent set \Emph{$I$}
in polynomial time that is clean for both $I_s$ and $I_t$. 
If each of $I_s$ and $I_t$ can reach a size-$k$ subset of $I$, either greedily or as in the proof of
\Cref{lem:atmost8}, then $(G,I_s,I_t)$ is a \textsf{yes}-instance, and we are done.  Note that we can
easily check this in polynomial time.  So instead we assume, by symmetry, that $I_s$ cannot do this.
Thus, \Cref{lem:atmost8} implies that $|V(G)\setminus N[I_s]|\le 2$.

    We bound the number of vertices in $G$ by grouping them according to the number of neighbors they have in $I_s$.  Recall that $|\mathcal{C}_3|\le 8k$; this handles all vertices with 3 or more neighbors in $I_s$.  And 
    as above $|V(G)\setminus N[I_s]|\le 2$.  So we focus on those vertices with exactly 1 or 2 neighbors in $I_s$.  Those with 1 neighbor in $I_s$ and 0 neighbors in $I_t$ are handled by the bound $|\mathcal{C}_1|\le 4k$; in fact, the same proof gives $|\mathcal{C}_1|\le 4k-3$, so we assume $|\mathcal{C}_1|+|V(G)\setminus N[I_s]|\le 4k$.  Thus, we are mainly concerned with vertices in $2$-classes, particularly those with at least 1 key vertex in $I_s$.
    As we noted in \Cref{cor:condition}, if at least $q$ $2$-classes each have at least one key vertex outside $I_s$ (for example in $I_t$), 
    then the number of vertices in these classes is at most $2q+2k$; otherwise, $I$ is $I_s$-greedy. 
    
    So we now restrict our attention to $2$-classes for which both key vertices are in $I_s$.  (In the final paragraph of the proof, we briefly return to vertices in $2$-classes with only 1 key vertex in $I_s$.)
    By the proof of \Cref{planar-complexity}, the number of $2$-classes with both key vertices in $I_s$ is at most $3k$. 
    We now delete some vertices from some of these $2$-classes.
    We will not modify any $2$-class of size at most $4$. We call a vertex $x$ of $I_s$     
    \Emph{important} if $x$ is a key vertex either for at least two $2$-classes of size at least $5$ or for (at least) one 
$2$-class of size at least $7$; otherwise, $x$ is \EmphE{unimportant}{2mm}. 
    We run the following algorithm: color blue the $2$-classes of size at least $7$ and color red the remaining $2$-classes 
    of size at least $5$. 
Now we iteratively remove colors from classes as long as every important vertex remains a key vertex 
    for either (at least) one blue class or two red classes; we continue removing colors until no color removal is possible.  For example, if a class is red and each of its key vertices is a key vertex for at least 3 red classes, then we can uncolor the class.  But if one of its key vertices is a key vertex for only 2 red classes (and for no blue classes), then we cannot uncolor the class, due to this key vertex.  We consider the classes in an arbitrary order, since the arguments below only require that our final red/blue/uncolored coloring of the classes is ``minimal'', nothing more.
    
    Denote by \Emph{$n_1, n_2, n_3$} the numbers of unimportant vertices,
    important vertices adjacent to a blue class, and important vertices adjacent only to red classes.  Note that $n_1+n_2+n_3=k$,
    since each vertex of $I_s$ is counted by precisely one $n_i$.
    When the algorithm stops, no color can be removed. In particular, every red class has a key vertex that is itself adjacent 
    to exactly $2$ red classes and to $0$ blue classes (otherwise the color can be removed). 

We now apply the following three reduction rules; we call the resulting graph \Emph{$G'$}.
\medskip

\noindent\textbf{Reduction Rule 1.} For every uncolored $2$-class $D$ for which both key vertices are important, delete all the vertices of $D$ except the vertices of $I$.
\medskip

\noindent\textbf{Reduction Rule 2.} For every $2$-class $C$ colored blue, delete all the vertices of $C$ that are not in $I$,
except for at most $7$.

\medskip

\noindent\textbf{Reduction Rule 3.}  For every $2$-class $C$ colored red, delete all the vertices of $C$ that are not in $I$, 
except for at most $5$ (this means removing at most 1 vertex).

\begin{clm}
    Reduction Rules 1, 2, and 3 are safe.  That is, $(G',I_s,I_t)$ and $(G,I_s,I_t)$ are equivalent.
\end{clm}
\begin{proof}
If a transformation exists from $I_s$ to $I_t$ in the reduced graph $G'$, then one also exists in the original graph $G$. 
We now prove the converse. 
    \smallskip
    
\noindent
    \textit{Rule 1.} 
Consider a transformation $\mathcal{R}$ from $I_s$ to $I_t$ in $G$. We prove that 
we can transform $I_s$ into $I$ in $G'$. By symmetry\footnote{This observation merits a bit more explanation, since we defined 
important vertices (and hence Reduction Rules 1, 2, and 3) w.r.t.~$I_s$, but not w.r.t.~$I_t$.  However, we simply note that when
we run $\mathcal{R}$ in reverse, from $I_t$ to $I_s$, the analysis in the following two paragraphs still holds.}, we can also transform $I_t$ to $I$; so in $G'$ we can transform $I_s$ into 
$I_t$, via $I$. 
    
    We assume that at some point in $\mathcal{R}$, there is a token on a vertex of $D$ (otherwise we are done, trivially).
    So at some step of the transformation, $\mathcal{R}$ unlocks $D$.
    When $D$ becomes unlocked, the token has moved off of one of its key vertices, which we denote by \Emph{$y$}. 
    Since $D$ is uncolored, and $y$ is an important vertex, $y$ is adjacent to either two red classes or to one blue class.
    So when $D$ is unlocked, we have additionally unlocked either a class of size at least $7$ (in $G'$) or two classes of 
    size at least $5$ (in $G'$) with $y$ as a key vertex.   

    So by \Cref{lem:reduc-equiv} 
    we can transform $I_s$ (and $I_t$) into an independent subset of $I$ (of size $k$) using only vertices of $I$, vertices of 
    colored classes, and vertices in classes adjacent to unimportant vertices. 
    Such a transformation does not use vertices of $D\setminus I$, so deleting them does not affect the existence of 
    such a transformation; this completes the proof.
    \smallskip

\noindent
\textit{Rules 2 and 3.} If some transformation from $I_s$ into $I_t$ uses a vertex of $C$, then, by 
\Cref{lem:reduc-equiv}, 
there exists such a transformation where the only vertices of $C$ used are in $I$ or in a 
    subset that we kept of size at most $7$ (resp.~$5$), corresponding respectively to Rules 2 and 3.  
\end{proof}

We now complete the proof, with a final counting argument. 
By Reduction Rule 1, each uncolored class with at least 5 vertices (in $G'$) outside $I$ has at least one key vertex that is unimportant. 
Since each unimportant vertex is the key vertex of at most one such class, the total number of uncolored classes with at least 5 
vertices outside $I$ is at most the number of unimportant vertices, $n_1$.  Thus, the total number of vertices outside $I$ in 
these (big, uncolored) classes is at most $6n_1$.  
Each blue class has a key vertex that is important, but that is not a key vertex for any other blue class.  
Thus, the total number of vertices, outside of $I$, in these blue classes is at most $7n_2$.  
And each red class has a key vertex that is important and is a key vertex for exactly one other
red class (and no blue class).  So the total number of vertices, outside of $I$, in red classes is at most $2\cdot 5n_3$.

Let $q$ be the total number of 2-classes with both key vertices in $I_s$; by the proof of 
\Cref{planar-complexity},
we have $q\le 3k$.  Let $u$, $b$, and $d$ be the numbers of these 2-classes that are, respectively, uncolored,
blue, and red.
So the total number of vertices in such $2$-classes, excluding the vertices of $I$, 
is at most $4q+2u+3b+d\le 4q+2n_1+3n_2+2n_3\le 4q+3k$.
There are at most $6k$ $2$-classes in total, so the number of $2$-classes that do not have both key vertices in $I_s$ is at most $6k-q$; by \Cref{cor:condition} the number of vertices in these classes is at most $2(6k-q)+2k=14k-2q$. 
So the total number of vertices in $2$-classes (outside $I$) 
is at most $(14k-2q)+(4q+3k) = 17k+2q \le 23k$.
So, after the application of Reduction Rules 1, 2, and 3, the number of vertices in $2$-classes, excluding vertices of $I$, is at most $23k$. Since $I$ has at most $2k$ vertices, the result holds.
\end{proof}

\subsection{Weakly Clean Independent Sets}\label{sec:weaklyclean}

In this section, we prove \Cref{lem:reduc-equiv}.  For clarity of presentation, we split the proof into $2$ parts, first 
removing vertices in each $2$-class of size at least $7$, and second removing vertices in each $2$-class of size at least $5$ 
that shares a key vertex with another $2$-class of size at least $5$.  It is straightforward to check that together 
\Cref{lem:locked_classes2} and \Cref{lem:locked_classes3} imply \Cref{lem:reduc-equiv}.
(Starting from $G$, we first apply \Cref{lem:locked_classes2} to get a subgraph $G''$; after this we apply 
\Cref{lem:locked_classes3} to $G''$ to get a graph $G'$.)
We use the following definitions.
A class $C$ is \EmphE{$I_s$-unlocked}{-4mm} (resp.~\emph{$I_t$-unlocked}) if $|N(C) \cap I_s| \le 1$ (resp.~$|N(C) \cap I_t|\le 1$). A class that is not $I_s$-unlocked is \Emph{$I_s$-locked}.

Recall that a vertex is \Emph{important} if it is a key vertex for either (a) at least one $2$-class of size at least $7$ or 
(b) at least two $2$-classes, each of size at least $5$.

\begin{lemma}\label{lem:locked_classes2}
    Let $G$ be a planar graph and $I$ be an independent set that is weakly clean for $I_s$.
    If we can unlock from $I_s$ a $2$-class of size at least $7$, then we can transform $I_s$ into any (size $k$) independent subset 
of $I$.
    Moreover, such a transformation still exists in every subgraph $G'$ formed from $G$ by deleting vertices (outside of $I$)
in $2$-classes of size at least $7$ such that every vertex important in $G$ remains important in $G'$.
\end{lemma}
\begin{proof}
    We first assume that $I$ is $3$-clean. We will explain at the end of the proof how to adapt our proof for $(2,2)$-clean 
independent sets.  The transformation we construct below is illustrated in Figure~\ref{fig:lemma_classes2}.
    Let $\{i_1,i_2\}$\aside{$i_1,i_2,i_3$} and $\{j_1,j_2\}$\aaside{$j_1,j_2$}{4mm} be the activation pairs, respectively, 
    of $I$ and $I_s$ and let $i_3$ be the auxiliary activation vertex of $I$. Recall, by Lemma~\ref{lem:hardcase}, that we can assume 
    the key vertices of the class $C$ containing $i_1,i_2,i_3$ are in $I_s$ (and thus are $j_1,j_2$).
    Our goal in this proof is to show that we can replace $\{j_1,j_2\}$ by $\{i_1,i_2\}$. If we can do this, then the conclusion 
    follows immediately, since $I$ is weakly $I_s$-greedy.

    Among all $2$-classes of size at least $7$ that can be unlocked from $I_s$, we choose a class \Emph{$D$} that can be unlocked 
    in the minimum number of steps. We denote by $x,y$ the two vertices of $N(D)\cap I_s$.
    (In the rest of the proof, we assume that $\{x,y\}$ and $\{j_1,j_2\}$ are disjoint; if not, then we simply omit useless 
    moves in the proposed transformation).  We consider a shortest transformation $\mathcal{R}$ from $I_s$ into an independent set 
    \Emph{$I'$} where $D$ is unlocked (the transformation $\mathcal{R}$ might be empty if $D$ is initially $I_s$-unlocked{; for now we assume that $\mathcal{R}$ is non-empty, but we remark briefly on the other case near the end of the proof}).  As we already 
    observed in the proof of \Cref{lem:unlock}: (i) the last step of $\mathcal{R}$ must consist in moving a token on $\{x,y\}$, 
    by symmetry we call it $y$, to another vertex, call it \Emph{$y'$}, and, (ii) each jump of the transformation except the last one 
    consists of moving a token (from its current vertex) to an adjacent vertex. 

% figure 4
\begin{figure}[!h]
    \centering
\begin{tikzpicture}[yscale=1.3, every edge quotes/.style = {outer sep=1pt, above, font=\footnotesize, draw=none, fill=white, fill opacity = .9}]
    \foreach \i in {1,2,3,4,5,6,9,10,11}
    {
        \draw  (\i,0) node (y\i) {};  % draw nodes
    }
	\draw (3,1) node (x3) {} (5,1) node (x5) {} (9.5,1) node (x9) {} (10.5,1) node (x10) {};
	\draw (-.5,.7) node (x0) {} 
	(y4) ++ (0,4*\offset) node[lStyle] (y4p) {}
	(y9) ++ (-2*\offset,-2*\offset) node[lStyle] (y9p) {};
	\draw[thick] (y1) -- (y2) -- (y3) -- (y4) -- (y5) -- (y6);

    \foreach \from/\to/\lab in {y2/y1/4, y4/y3/5, y6/y5/6}
	\draw[thick, mygreen] (\from) edge[bend right=30, ->, shorten <= 0.8mm, shorten >= .8mm, "\lab"] (\to); % draw all arrows

     \draw[thick, mypurple] (x9) edge[bend right=10, ->, shorten <= 0.8mm, shorten >= .8mm, "3"] (y6); 
     \draw[thick, mypurple] (x5) edge[bend right=20, ->, shorten <= 0.8mm, shorten >= .8mm, "0"] (x0); 
     \draw[thick, mypurple] (x0) edge[bend right=0, ->, shorten <= 0.8mm, shorten >= .8mm, "2"] (y4p); 
	\begin{scope}[every edge quotes/.style = {outer sep=1.5pt, right, draw=none, fill=white, fill opacity=.9, font=\footnotesize}] 
		\draw[thick, mypurple] (x3) edge[bend right=10, ->, shorten <= 0.8mm, shorten >= .8mm, "1"] (y2); 
	\end{scope}
	\begin{scope}[every edge quotes/.style = {outer sep=6.0pt, right, draw=none, fill=white, fill opacity=.9, font=\footnotesize}] 
       		\draw[thick, myblue] (y5) edge[bend right=10, ->, shorten <= 0.8mm, shorten >= .8mm, "10"] (x3); 
	\end{scope}
	\begin{scope}[every edge quotes/.style = {outer sep=2.75pt, right, pos = .5, draw=none, fill=none, fill opacity=.9, font=\footnotesize}] 
       		\draw[thick, myblue] (x10) edge[ ->, shorten <= 0.8mm, shorten >= .8mm, "7"] (y9p); 
	\end{scope}
	\begin{scope}[every edge quotes/.style = {outer sep=1.0pt, below, draw=none, fill=white, fill opacity=.9, font=\footnotesize}] 
       		\draw[thick, myblue] (y10) edge[bend right=5, ->, shorten <= 0.8mm, shorten >= .8mm, "11"] (x5); 
       		\draw[thick, myblue] (y3) edge[bend right=15, ->, shorten <= 0.8mm, shorten >= .8mm, "9"] (y10); 
       		\draw[thick, myblue] (y1) edge[bend right=20, ->, shorten <= 0.8mm, shorten >= .8mm, "8"] (y11); 
	\end{scope}

	\foreach \where/\lab in {x0/{y'}, x3/x, x5/y, x9/{j_1}, x10/{j_2}}
	\draw (\where) ++ (0,13*\offset) node[lStyle] {\footnotesize{$\lab$}};

	\foreach \where/\lab in {y1/{v_1'}, y2/{v_1}, y3/{v_2'}, y4/{v_2}, y5/{v_3'}, y6/{v_3}}
	\draw (\where) ++ (0,-14*\offset) node[lStyle, fill=white, opacity=.9, shape=rectangle] {\footnotesize{$\lab$}};

	\foreach \where/\lab in {y9/{i_1}, y10/{i_3}, y11/{i_2}}
	\draw (\where) ++ (0,-12.5*\offset) node[lStyle, fill=white, opacity=.9, shape=rectangle] {\footnotesize{$\lab$}};
    \begin{scope}[yshift=-1.05cm]
    \draw[semithick, rounded corners=3.25mm, gray] (1.5,.68) -- (6.3,.68) -- (6.3,1.265) -- (.7,1.265) -- (.7, .68) -- (1.5,.68);
	    \draw (0.45,1) node[lStyle] {\footnotesize{$D$}};
    
    \begin{scope}[xshift=3.15in]
    	\draw[semithick, rounded corners=3.25mm, gray] (1.5,.68) -- (3.3,.68) -- (3.3,1.265) -- (.7,1.265) -- (.7, .68) -- (1.5,.68);
	    \draw (3.55,1) node[lStyle] {\footnotesize{$C$}};
    \end{scope}
    \end{scope}
\end{tikzpicture}
\captionsetup{width=.4\textwidth}
    \caption{
	    \textcolor{mypurple}{\textbf{---} phase 1}
	    \textcolor{mygreen}{\textbf{---} phase 2}
	    \textcolor{myblue}{\textbf{---} phase 3}\\
The proof of Lemma~\ref{lem:locked_classes2} when $I$ is $3$-clean.
    }
    \label{fig:lemma_classes2}
\end{figure}

We denote by \Emph{$Z$} the set of vertices that were involved in a move during the sequence $\mathcal{R}$, 
excluding vertex $y$.  Here a vertex $v$ is \Emph{involved in a move} if there exist two steps of the sequence 
such that one has a token on $v$ and one has no token on $v$.  By (i) and (ii) above, all the vertices of 
$Z$ belong to the same component of $G \setminus (D \cup \{ x,y\})$.  
(If $Z$ contains vertices in a component other than that containing $y'$, the we can simply omit all moves
in that component.  Since all moves but the last are to adjacent vertices, this subsequence still frees $y'$
to receive a token from $y$. Finally, this shorter sequence contradicts the minimality of $\mathcal{R}$.)
So, by \Cref{lem:minor2vertices}, the set $N(Z) \cap D$ has size at most $2$. 
Since $|D|\ge 7$, there exists $D'\subseteq D$ with $|D'|\ge 5$ and $N(Z)\cap D'=\emptyset$.  
So \Cref{lem:deletion2perclass} ensures that $D$ contains an independent set $v_1,v_2,v_3$ that is 
anticomplete to $I' \setminus I_s$. 

Recall that at this point $D$ is unlocked (since $y$ is not in $I'$) and $\{v_1,v_2,v_3\}$ is anticomplete to 
$I' \setminus I_s$.  We can now move $x \rightarrow v_1$, $y' \rightarrow v_2$, and $j_1 \rightarrow v_3$. 
(If $C=D$, then we simply omit the third move of this sequence. {And if the token initially on $j_1$ has moved, which might have happened since $C$ might have size less than $7$, then we still move that token to $v_3$; this is possible because $v_1,v_2,v_3$ are chosen to be anticomplete to all vertices that received a token during $\mathcal{R}$.}) We call this part of the transformation 
Phase 1\aside{Phase 1}; see Figure~\ref{fig:lemma_classes2}.
    
After this sequence of moves, we reach an independent set $J$ containing only one key vertex of $C$. 
We would like to replace $j_2$ with a vertex of $\{i_1,i_2,i_3\}$.  However, we cannot do this directly, 
since vertices of the current independent set might be adjacent to $\{i_1,i_2,i_3\}$.  To allow these moves, 
we first cancel all the moves of $\mathcal{R}$ (except the last one which moved the token initially on $y$) 
by applying $\mathcal{R}^{-1}$, the moves of $\mathcal{R}$ in reverse order.  Since all vertices involved in 
moves of $\mathcal{R}$, and thus of $\mathcal{R}^{-1}$, are anticomplete to $\{v_1,v_2,v_3\}$, applying 
$\mathcal{R}^{-1}$ does not create any conflict; that is, we keep an independent set throughout the 
transformation. 

Note that, at this point, we have transformed $I_s$ into $(I_s \setminus \{x,y,j_1\}) \cup \{v_1,v_2,v_3\}$. If $C=D$, then to conclude we simply need to move tokens from $\{v_1,v_2\}$ to $\{i_1,i_2\}$. This is indeed possible, since $G[D]$ is a cycle or a disjoint union of paths, and we can always transform an independent set of size at most $3$ on a path/cycle into another as long as 
$|D|\ge 7$.
Otherwise,  \Cref{lem:deletion2perclass} ensures that $D$ contains an independent set $\{v_1',v_2',v_3'\}$ anticomplete to $C$ (and 
thus to $\{i_1,i_2,i_3\}$).  So we can transform $(I_s \setminus \{x,y,j_1\}) \cup \{v_1,v_2,v_3\}$ into $(I_s \setminus \{x,y,j_1\}) \cup \{v_1',v_2',v_3'\}$; this is Phase 2.\aside{Phase 2}

    Now we finally use the moves $j_2 \rightarrow i_1$, $v_1' \rightarrow i_2$, $v_2' \rightarrow i_3$, $v_3' \rightarrow x$, and 
$i_3 \rightarrow y$; this is Phase 3.\aside{Phase 3} We have thus reached the independent set $(I_s\setminus\{j_1,j_2\})\cup\{i_1,i_2\}$.
Since $I$ is weakly $I_s$-greedy, we can transform $I_s$ into $I$; this proves the lemma.

    {Finally, if the transformation $\mathcal{R}$ is empty because the class $D$ was initially unlocked, then we can perform the same sequence of moves except that there is no token initially on $y$ so all the moves corresponding to that token can be ignored.}
\begin{figure}[!t]
    \centering
\begin{tikzpicture}[yscale=1.3, every edge quotes/.style = {outer sep=1pt, above, font=\footnotesize, draw=none, fill=white, fill opacity = .9}]
    \foreach \i in {1,2,3,4,5,6,7,9,10,11,12}
    {
        \draw  (\i,0) node (y\i) {};  % draw nodes
    }
    \draw (3,1) node (x3) {} (5,1) node (x5) {} (9.5,1) node (x9) {} (10.5,1) node (x10) {} (11.5,1) node (x11) {};
	\draw (-.5,.7) node (x0) {} 
	(y4) ++ (0,4*\offset) node[lStyle] (y4p) {}
	(y9) ++ (-2*\offset,-2*\offset) node[lStyle] (y9p) {};
    \draw[thick] (y1) -- (y2) -- (y3) -- (y4) -- (y5) -- (y6) -- (y7);

    \foreach \from/\to/\lab in {y2/y1/4, y4/y3/5, y6/y5/6}
	\draw[thick, mygreen] (\from) edge[bend right=30, ->, shorten <= 0.8mm, shorten >= .8mm, "\lab"] (\to); % draw all arrows

     \draw[thick, mypurple] (x10) edge[bend right=0, ->, shorten <= 0.8mm, shorten >= .8mm, "3"] (y6); 
     \draw[thick, mypurple] (x5) edge[bend right=20, ->, shorten <= 0.8mm, shorten >= .8mm, "0"] (x0); 
     \draw[thick, mypurple] (x0) edge[bend right=0, ->, shorten <= 0.8mm, shorten >= .8mm, "2"] (y4p); 
	\begin{scope}[every edge quotes/.style = {outer sep=1.5pt, right, draw=none, fill=white, fill opacity=.9, font=\footnotesize}] 
		\draw[thick, mypurple] (x3) edge[bend right=10, ->, shorten <= 0.8mm, shorten >= .8mm, "1"] (y2); 
	\end{scope}
	\begin{scope}[every edge quotes/.style = {outer sep=1.0pt, below left, draw=none, fill=white, fill opacity=.9, font=\footnotesize}] 
       		\draw[thick, myblue] (y7) edge[bend right=10, ->, shorten <= 0.8mm, shorten >= .8mm, "12"] (x3); 
	\end{scope}
	\begin{scope}[every edge quotes/.style = {outer sep=0.0pt, right, pos = .5, draw=none, fill=none, fill opacity=.9, font=\footnotesize}] 
       		\draw[thick, myblue] (x9) edge[->, shorten <= 0.8mm, shorten >= .8mm, "7"] (y9p); 
	\end{scope}
	\begin{scope}[every edge quotes/.style = {outer sep=1.0pt, below, draw=none, fill=white, fill opacity=.9, font=\footnotesize}] 
       		\draw[thick, myblue] (y1) edge[pos = .3, bend right=20, ->, shorten <= 0.8mm, shorten >= .8mm, "8"] (y10); 
       		\draw[thick, myblue] (y5) edge[pos = .5, bend right=32, ->, shorten <= 0.8mm, shorten >= .8mm, "9"] (y7); 
       		\draw[thick, myblue] (y3) edge[pos = .7, bend right=20, ->, shorten <= 0.8mm, shorten >= .8mm, "11"] (y12); 
	\end{scope}
       		\draw[thick, myblue] (y12) edge[out=60, in = 30, ->, shorten <= 0.8mm, shorten >= .8mm, "13"] (x5); 
	\begin{scope}[every edge quotes/.style = {pos = .8, outer sep=0.3pt, left, draw=none, fill=none, font=\footnotesize}] 
       		\draw[thick, myblue] (y11) edge[bend left = 20, ->, shorten <= 0.8mm, shorten >= .8mm, "14"] (x11); 
	\end{scope}
	\begin{scope}[every edge quotes/.style = {pos = .95, outer sep=1.3pt, right, draw=none, fill=none, font=\footnotesize}] 
       		\draw[thick, myblue] (x11) edge[left, bend left = 20, ->, shorten <= 0.8mm, shorten >= .8mm, "10"] (y11); 
	\end{scope}

    \foreach \where/\lab in {x0/{y'}, x3/x, x5/y, x9/{j_2}, x10/{j_1}, x11/{j_3}}
	\draw (\where) ++ (0,13*\offset) node[lStyle] {\footnotesize{$\lab$}};

    \foreach \where/\lab in {y1/{v_1'}, y2/{v_1}, y3/{v_2'}, y4/{v_2}, y5/{v_3'}, y6/{v_3}, y7/{~v_3''}}
	\draw (\where) ++ (0,-14*\offset) node[lStyle, fill=none, opacity=.9, shape=rectangle] {\footnotesize{$\lab$}};

    \foreach \where/\lab in {y9/{i_1}, y10/{i_2}, y11/{i_3}, y12/{i_4}}
	\draw (\where) ++ (0,-12.5*\offset) node[lStyle, fill=white, opacity=.9, shape=rectangle] {\footnotesize{$\lab$}};
    \begin{scope}[yshift=-1.05cm]
    \draw[semithick, rounded corners=3.25mm, gray] (1.5,.68) -- (7.35,.68) -- (7.35,1.265) -- (.7,1.265) -- (.7, .68) -- (1.5,.68);
	    \draw (0.45,1) node[lStyle] {\footnotesize{$D$}};
    
    \begin{scope}[xshift=3.15in]
    	\draw[semithick, rounded corners=3.25mm, gray] (1.5,.68) -- (2.3,.68) -- (2.3,1.265) -- (.7,1.265) -- (.7, .68) -- (1.5,.68);
	    \draw (0.5,1) node[lStyle] {\footnotesize{$C$}};
    \end{scope}
    
    \begin{scope}[xshift=3.93in]
    	\draw[semithick, rounded corners=3.25mm, gray] (1.5,.68) -- (2.3,.68) -- (2.3,1.265) -- (.7,1.265) -- (.7, .68) -- (1.5,.68);
	    \draw (2.55,1) node[lStyle] {\footnotesize{$C'$}};
    \end{scope}

    \end{scope}
\end{tikzpicture}
\captionsetup{width=.4\textwidth}
    \caption{
	    \textcolor{mypurple}{\textbf{---} phase 1}
	    \textcolor{mygreen}{\textbf{---} phase 2}
	    \textcolor{myblue}{\textbf{---} phase 3}\\
The proof of Lemma~\ref{lem:locked_classes2} when $I$ is $(2,2)$-clean.
    }
    \label{fig:lemma_classes2b}
\end{figure}

\paragraph{Modification for $\bm{(2,2)}$-clean independent sets.}
For $(2,2)$-clean independent sets, we will need to use additional activation vertices. We let $C'$ be the class containing the 
two additional auxiliary vertices $i_3,i_4$, and let $j_3$ be the auxiliary activation vertex of $I_s$. We assume that $j_1$ is a 
key vertex of both classes $C$ and $C'$; see \Cref{fig:lemma_classes2b}.  The first phase of the transformation is identical. 

If $C=D$, then we move the tokens from $\{v_1,v_2\}$ to $\{i_1,i_2\}$, and are done.  So we assume $C \ne D$. By~\Cref{lem:deletion2perclass} with $C_1=D$ and $C_r=C$, 
and since $|D| \ge 7$, there exists an independent subset $\{v_1',v_2',v_3'\}$ of $D$ that is anticomplete to $\{i_1,i_2\}$. We apply Phase 2 to move the tokens from $\{v_1,v_2,v_3\}$ to $\{v_1',v_2',v_3'\}$. We then move $j_2 \rightarrow i_1$, $v_1' \rightarrow i_2$. At most two vertices of $D$ are adjacent to $C$, and at most two are adjacent to $C'$; so there exist three vertices in $D$ not 
adjacent to $C \cup C'$. Since $G[D]$ is either a cycle of length at least 7 or a disjoint union of paths, 
this set of size $3$ has an independent subset of size two, 
$\{v_2'',v_3''\}$.  We move the tokens from $\{v_2',v_3'\}$ to $\{v_2'',v_3''\}$. Finally, we perform the moves $j_3 \rightarrow 
i_3$, $v_2'' \rightarrow i_4$, $v_3'' \rightarrow x$, $i_4 \rightarrow y$, and $i_3 \rightarrow j_3$; this completes the proof. 
\end{proof}

The following lemma has a proof with a similar flavor.

\begin{lemma}
\label{lem:locked_classes3}
    Let $G$ be a planar graph and $I$ be an independent set that is weakly clean for $I_s$.
If we can unlock from $I_s$ a $2$-class of size at least $5$ that has a key vertex $x$ that is also a key vertex of another $2$-class 
also of size at least $5$, then we can transform $I_s$ into any (size $k$) independent subset of $I$.
    Moreover, such a transformation still exists in every subgraph $G'$ formed from $G$ by deleting vertices (outside of $I$)
in $2$-classes of size at least $5$ (each with a key vertex that is also a key vertex for another $2$-class of size at least $5$) 
such that every vertex important in $G$ remains important in $G'$.
\end{lemma}
\begin{proof}
    The transformation we construct below is illustrated in Figures~\ref{fig:lemma_classes2c} and~\ref{fig:lemma_classes2d}.
    We first assume that $I$ is $3$-clean. At the end of the proof, we will explain how to adapt the proof for 
independent sets that are $(2,2)$-clean.
Let $\{i_1,i_2\}$ and $\{j_1,j_2\}$ be the activation pairs, respectively, of $I$ and $I_s$, and let $i_3$ be the 
    auxiliary activation vertex of $I$.  We denote by $C$ the $2$-class containing $\{i_1,i_2,i_3\}$.
    The proof consists in showing that we can replace $\{j_1,j_2\}$ in $I_s$ by $\{i_1,i_2\}$. If we can, then the conclusion 
    follows directly, since $I$ is weakly $I_s$-greedy.

    Among all the key vertices $x$ adjacent to at least two $2$-classes of size at least $5$ such that at least one of them can 
    be unlocked from $I_s$, we choose a vertex $x$ and a class $A$ with key vertex $x$ that can be unlocked with the minimum number 
    of steps.  We consider a shortest transformation $\mathcal{R}$ from $I_s$ into an independent set $I'$ where $A$ is unlocked.
    We denote by $B$ another $2$-class of size at least $5$ for which $x$ is a key vertex. We denote by $y$ and $y'$ the second 
    key vertices, respectively, of $A$ and $B$. Even if $x$ and $y$ are not symmetric, in what follows, we do not 
    care which of $x$ or $y$ has been moved when $A$ is unlocked; so we assume that the token on $y$ has been moved.

% figure 5
\begin{figure}[!t]
    \centering
\begin{tikzpicture}[yscale=1.3, every edge quotes/.style = {outer sep=1pt, above, font=\footnotesize, draw=none, fill=white, fill opacity = .9}]
    \clip (-.5,-1.5) rectangle + (16.5,3.9);  
    \foreach \i in {1,2,3,4,6,7,8,9,11,12,13}
    {
        \draw  (\i,0) node (x\i) {};  % draw nodes
    }
    \draw (4,1) node (y) {} (5,1) node (x) {} (6,1) node (y') {} (2,1) node (unnamed) {} (11.5,1) node (j1) {} (12.5,1) node (j2) {};
	\draw 
	(y') ++ (0,2*\offset) node[lStyle] (y'p) {} 
    (y) ++ (0,-2*\offset) node[lStyle] (yp) {}
    (x1) ++ (0,-3*\offset) node[lStyle] (x1p) {}
    (x6) ++ (3*\offset,0) node[lStyle] (x6p) {}
    (x8) ++ (0,2*\offset) node[lStyle] (x8p) {};
    \draw[thick] (x1) -- (x2) -- (x3) -- (x4)  (x6) -- (x7) -- (x8) -- (x9);

	\begin{scope}[every edge quotes/.style = {outer sep=0.2pt, below, draw=none, fill=none, font=\footnotesize}] 
    \foreach \from/\to/\lab in {x1/x2/6, x3/x4/5, x6/x7/8, x8/x9/7}
	\draw[thick, mygreen] (\from) edge[pos = .5, bend right=20, below, ->, shorten <= 0.8mm, shorten >= .8mm, "\lab"] (\to); % draw all arrows
     \draw[thick, myblue] (x2) edge[bend right=20, pos = .5, ->, shorten <= 0.8mm, shorten >= .8mm, "11"] (x13); 
    \end{scope}
	\begin{scope}[every edge quotes/.style = {outer sep=3.0pt, left, draw=none, fill=none, font=\footnotesize}] 
     \draw[thick, myblue] (x7) edge[pos = .4, bend right=0, ->, shorten <= 0.8mm, shorten >= .8mm, "13"] (x); 
    \end{scope}
	\begin{scope}[every edge quotes/.style = {outer sep=12.0pt, right, draw=none, fill=none, font=\footnotesize}] 
     \draw[thick, myblue] (x9) edge[bend right=0, ->, shorten <= 0.8mm, shorten >= .8mm, "~12"] (yp); 
    \end{scope}

	\begin{scope}[every edge quotes/.style = {outer sep=0.5pt, above, draw=none, fill=none, font=\footnotesize}] 
     \draw[thick, mypurple] (x) edge[pos=.4, bend right=0, ->, shorten <= 0.8mm, shorten >= .8mm, "1"] (x3); 
     \draw[thick, mypurple] (y') edge[bend right=30, out = -45, in = -90, ->, shorten <= 0.8mm, shorten >= .8mm, "2"] (x1p); 
     \draw[thick, mypurple] (j1) edge[bend right=0, ->, shorten <= 0.8mm, shorten >= .8mm, "3"] (x8p); 
     \draw[thick, mypurple] (unnamed) edge[pos = 0.7, bend right=0, ->, shorten <= 0.8mm, shorten >= .8mm, "4"] (x6p); 
     \draw[thick, mypurple] (y) edge[bend right=0, ->, shorten <= 0.8mm, shorten >= .8mm, "0"] (unnamed); 
     \draw[thick, myblue] (j2) edge[pos = .5, bend right=0, ->, shorten <= 0.8mm, shorten >= .8mm, "9"] (x11); 
     \draw[thick, myblue] (x4) edge[ out = -60, in = -120, bend right=20,pos = .5, ->, shorten <= 0.8mm, shorten >= .8mm, "10"] (x12); 
     \draw[thick, myblue] (x13) edge[out = 60, in = 20,pos = .5, ->, shorten <= 0.8mm, shorten >= .8mm, "14"] (y'); 
    \end{scope}

    \foreach \where/\lab in {y'/y', x/x, y/y, j1/{j_1}, j2/{j_2}}
	\draw (\where) ++ (0,13*\offset) node[lStyle] {\footnotesize{$\lab$}};

    \foreach \where/\lab in {x1/{a_2}, x2/{a_2'}, x3/{a_1}, x4/{a_1'}, x6/{b_2}, x7/{b_2'}, x8/{b_1}, x9/{b_1'} }
	\draw (\where) ++ (0,-14*\offset) node[lStyle, fill=none, opacity=.9, shape=rectangle] {\footnotesize{$\lab$}};
    \foreach \where/\lab in {x11/{i_1}, x12/{i_2}, x13/{i_3}}
	\draw (\where) ++ (0,-12.5*\offset) node[lStyle, fill=white, opacity=.9, shape=rectangle] {\footnotesize{$\lab$}};
    \begin{scope}[yshift=-1.05cm]
    \draw[semithick, rounded corners=3.25mm, gray] (1.5,.68) -- (4.35,.68) -- (4.35,1.265) -- (.7,1.265) -- (.7, .68) -- (1.5,.68);
	    \draw (0.5,1) node[lStyle] {\footnotesize{$A$}};
    
    \begin{scope}[xshift=1.97in]
    	\draw[semithick, rounded corners=3.25mm, gray] (1.5,.68) -- (4.3,.68) -- (4.3,1.265) -- (.7,1.265) -- (.7, .68) -- (1.5,.68);
	    \draw (0.5,1) node[lStyle] {\footnotesize{$B$}};
    \end{scope}
    
    \begin{scope}[xshift=3.93in]
    	\draw[semithick, rounded corners=3.25mm, gray] (1.5,.68) -- (3.3,.68) -- (3.3,1.265) -- (.7,1.265) -- (.7, .68) -- (1.5,.68);
	    \draw (0.5,1) node[lStyle] {\footnotesize{$C$}};
    \end{scope}

    \end{scope}
\end{tikzpicture}
\captionsetup{width=.4\textwidth}
    \caption{
	    \textcolor{mypurple}{\textbf{---} phase 1}
	    \textcolor{mygreen}{\textbf{---} phase 2}
	    \textcolor{myblue}{\textbf{---} phase 3}
The proof of Lemma~\ref{lem:locked_classes3} when $I$ is $3$-clean.
    }
    \label{fig:lemma_classes2c}
\end{figure}
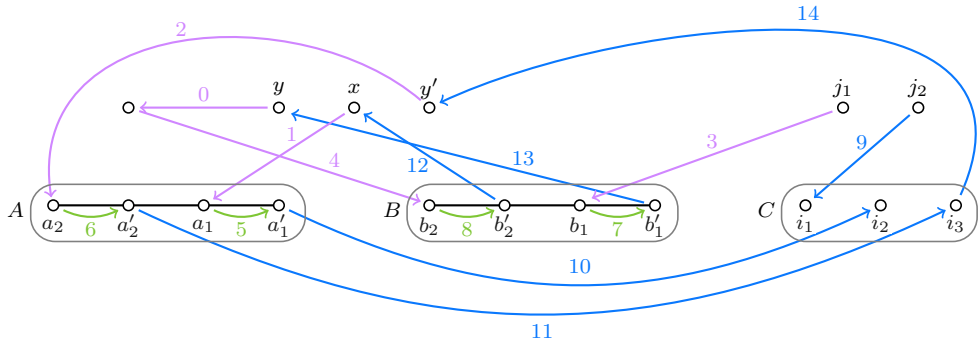

As in the proof of \Cref{lem:locked_classes2}, we can argue using \Cref{lem:unlock} that all the vertices 
$Z$ involved in the transformation between $I_s$ and $I'$ (except for the last vertex $y$ replaced in the 
sequence, which is a key vertex of $A$) induce a connected subgraph.  By applying 
\Cref{lem:deletion2perclass_upd}(a) to $Z,A,B$, we can find independent sets $\{a_1,a_2\}\subseteq A$ and 
$\{b_1,b_2\}\subseteq B$ that are anticomplete to each other and both anticomplete to vertices of 
 $I' \setminus I_s$.  In other words, 
the only occupied neighbor of $\{a_1,a_2\}$ is the remaining key vertex of $A$, and the occupied neighbors
of $\{b_1,b_2\}$ are the two key vertices of $B$.

    We now perform the moves $x \rightarrow a_1$, $y' \rightarrow a_2$, $j_1 \rightarrow b_1$, and then move 
    the token initially on $y$ to $b_2$. As in the proof of \Cref{lem:locked_classes2}, we cancel all these moves of $\mathcal{R}$   
    except the last, which moved the token off of $y$, by applying $\mathcal{R}^{-1}$. After this transformation, we reach the independent set $(I_s 
    \setminus \{ x,y,y',j_1 \}) \cup \{ a_1,a_2,b_1,b_2 \}$.

    By \Cref{lem:deletion2perclass}, we can find subsets of size $3$ in $A$ and $B$ that are anticomplete to each other and also 
    anticomplete to $C$. In particular, we can find independent sets $\{a_1',a_2'\}$ in $A$ and $\{b_1',b_2'\}$ in $B$ that are
    anticomplete to $\{i_1,i_2,i_3\}$. We can thus move the tokens on $\{a_1,a_2,b_1,b_2\}$ onto $\{a_1',a_2',b_1',b_2'\}$.

We finally perform the moves $j_2 \rightarrow i_1$, $a_1' \rightarrow i_2$, $a_2' \rightarrow i_3$, $b_1' \rightarrow y$, $b_2' \rightarrow x$, and $i_3 \rightarrow y'$. We have thus reached the independent set where the $I_s$-activation pair has been 
replaced by the $I$-activation pair. Since $I$ is weakly $I_s$-greedy, we can transform $I_s$ into $I$ and the conclusion follows.

\paragraph{Modification for $\bm{(2,2)}$-clean independent sets.}
For $(2,2)$-clean independent sets, we will need to use both auxiliary activation vertices. The first phase of the transformation 
is identical (we simply choose $j_1$ as the common key vertex of $C$ and $C'$). Recall that this transformation ends with 
$x\to a_1$, $y'\to a_2$, $j_1\to b_1$, and moving the token initially on $y$ to $b_2$.
When we select $a_1',a_2',b_1',b_2'$, we need to be careful. 
\smallskip

\begin{figure}[!h]
    \centering
\begin{tikzpicture}[yscale=1.3, every edge quotes/.style = {outer sep=1pt, above, font=\footnotesize, draw=none, fill=white, fill opacity = .9}]
    \clip (-.5,-1.5) rectangle + (16.5,3.9);  
    \foreach \i in {1,2,3,4,6,7,8,9,11,12,13,14}
    {
        \draw  (\i,0) node (x\i) {};  % draw nodes
    }
    \draw (4,1) node (y) {} (5,1) node (x) {} (6,1) node (y') {} (2,1) node (unnamed) {} (11.5,1) node (j2) {} 
    (12.5,1) node (j1) {} (13.5,1) node (j3) {};

	\draw 
	(y') ++ (0,2*\offset) node[lStyle] (y'p) {} 
    (y) ++ (0,-2*\offset) node[lStyle] (yp) {}
    (x1) ++ (0,-3*\offset) node[lStyle] (x1p) {}
    (x6) ++ (3*\offset,0) node[lStyle] (x6p) {}
    (x8) ++ (0,2*\offset) node[lStyle] (x8p) {};
    \draw[thick] (x1) -- (x2) -- (x3) -- (x4)  (x6) -- (x7) -- (x8) -- (x9);

	\begin{scope}[every edge quotes/.style = {outer sep=0.2pt, below, draw=none, fill=none, font=\footnotesize}] 
    \foreach \from/\to/\lab in {x1/x2/6, x3/x4/5, x6/x7/8, x8/x9/7}
	\draw[thick, mygreen] (\from) edge[pos = .5, bend right=20, below, ->, shorten <= 0.8mm, shorten >= .8mm, "\lab"] (\to); % draw all arrows
     \draw[thick, myblue] (x9) edge[out = -60, in = -120, bend right=25,pos = .5, ->, shorten <= 0.8mm, shorten >= .8mm, "10"] (x12); 
     \draw[thick, myblue] (x4) edge[bend right=20, pos = .5, ->, shorten <= 0.8mm, shorten >= .8mm, "12"] (x14); 
     \draw[thick, myblue] (j3) edge[bend right=0, pos = .2, ->, shorten <= 0.8mm, shorten >= .8mm, "~~11"] (x13); 
    \end{scope}

	\begin{scope}[every edge quotes/.style = {outer sep=0.5pt, above, draw=none, fill=none, font=\footnotesize}] 
     \draw[thick, mypurple] (x) edge[pos=.4, bend right=0, ->, shorten <= 0.8mm, shorten >= .8mm, "1"] (x3); 
     \draw[thick, mypurple] (y') edge[bend right=30, out = -45, in = -90, ->, shorten <= 0.8mm, shorten >= .8mm, "2"] (x1p); 
     \draw[thick, mypurple] (j1) edge[bend right=0, ->, shorten <= 0.8mm, shorten >= .8mm, "3"] (x8p); 
     \draw[thick, mypurple] (unnamed) edge[pos = 0.7, bend right=0, ->, shorten <= 0.8mm, shorten >= .8mm, "4"] (x6p); 
     \draw[thick, mypurple] (y) edge[bend right=0, ->, shorten <= 0.8mm, shorten >= .8mm, "0"] (unnamed); 
    \end{scope}
	\begin{scope}[every edge quotes/.style = {draw=none, fill=none, right, font=\footnotesize}] 
     \draw[thick, myblue] (j2) edge[pos = .6, bend right=0, ->, shorten <= 0.8mm, shorten >= .8mm, "9~~"] (x11); 
    \end{scope}

    \foreach \where/\lab in {y'/y', x/x, y/y, j1/{j_1}, j2/{j_2}, j3/{j_3}}
	\draw (\where) ++ (0,13*\offset) node[lStyle] {\footnotesize{$\lab$}};

    \foreach \where/\lab in {x1/{a_2}, x2/{a_2'}, x3/{a_1}, x4/{a_1'}, x6/{b_2}, x7/{b_2'}, x8/{b_1}, x9/{b_1'} }
	\draw (\where) ++ (0,-14*\offset) node[lStyle, fill=none, opacity=.9, shape=rectangle] {\footnotesize{$\lab$}};
    \foreach \where/\lab in {x11/{i_1}, x12/{i_2}, x13/{i_3}, x14/{i_4}}
	\draw (\where) ++ (0,-12.5*\offset) node[lStyle, fill=white, opacity=.9, shape=rectangle] {\footnotesize{$\lab$}};
    \begin{scope}[yshift=-1.05cm]
    \draw[semithick, rounded corners=3.25mm, gray] (1.5,.68) -- (4.35,.68) -- (4.35,1.265) -- (.7,1.265) -- (.7, .68) -- (1.5,.68);
	    \draw (0.5,1) node[lStyle] {\footnotesize{$A$}};
    
    \begin{scope}[xshift=1.97in]
    	\draw[semithick, rounded corners=3.25mm, gray] (1.5,.68) -- (4.3,.68) -- (4.3,1.265) -- (.7,1.265) -- (.7, .68) -- (1.5,.68);
	    \draw (0.5,1) node[lStyle] {\footnotesize{$B$}};
    \end{scope}

    \begin{scope}[xshift=3.93in]
    	\draw[semithick, rounded corners=3.25mm, gray] (1.5,.68) -- (2.3,.68) -- (2.3,1.265) -- (.7,1.265) -- (.7, .68) -- (1.5,.68);
	    \draw (0.5,1) node[lStyle] {\footnotesize{$C$}};
    \end{scope}
    
    \begin{scope}[xshift=4.725in]
    	\draw[semithick, rounded corners=3.25mm, gray] (1.5,.68) -- (2.3,.68) -- (2.3,1.265) -- (.7,1.265) -- (.7, .68) -- (1.5,.68);
	    \draw (2.5,1) node[lStyle] {\footnotesize{$C'$}};
    \end{scope}

    \end{scope}
\end{tikzpicture}
\captionsetup{width=.4\textwidth}
    \caption{
	    \textcolor{mypurple}{\textbf{---} phase 1}
	    \textcolor{mygreen}{\textbf{---} phase 2}
	    \textcolor{myblue}{\textbf{---} phase 3}\\
The proof of Lemma~\ref{lem:locked_classes3} when $I$ is $(2,2)$-clean.
        We omit the final few moves of the transition,\\ since they depend on the cases described below.
    }
    \label{fig:lemma_classes2d}
\end{figure}
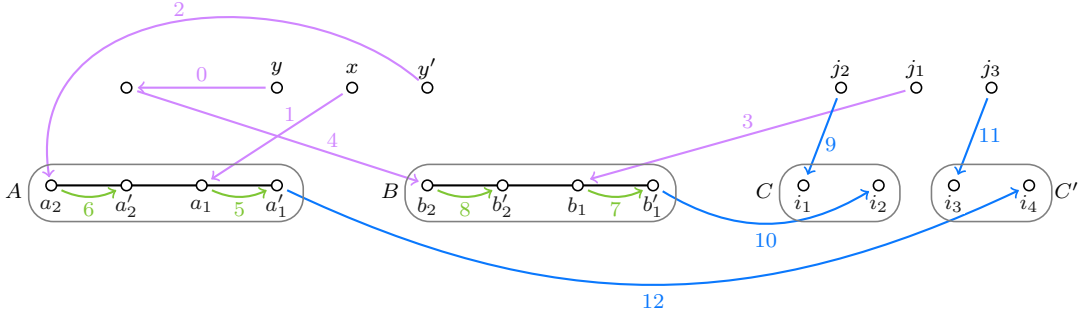

\noindent
{\bf Case 1: $\bm{x\notin \{j_1,j_2,j_3\}}$.} 
Applying Lemma~\ref{lem:deletion2perclass} to $A$, $B$, and $C$ ensures that there exist sets of three vertices in each of $A,B$ 
that are anticomplete to each other and both anticomplete to $C$.
So there exist $\{a_1',a_2'\}\subseteq A$ and $\{b_1',b_2'\}\subseteq B$ that are both anticomplete to $C$. 
Since $G[A]$ is a cycle of length at least 5 or disjoint union of paths, at most one of $a_1',a_2'$ is 
adjacent to both $a_1,a_2$.  So we can immediately move, say, $a_1$ to $a_1'$; afterwards we move $a_2$ to
$a_2'$.  A similar argument applies for moving $\{b_1,b_2\}$ to $\{b_1',b_2'\}$.
We perform the moves $j_2 \rightarrow i_1$ and $b_1' \rightarrow i_2$.  We consider $ A \cup B \cup \{ x \}$ as a connected component $W$.
Applying \Cref{lem:deletion2perclass_upd}(b) to $W$ and $C \cap I,C' \cap I$ ensures that by deleting $4$ vertices in $W$ we remove all edges to 
$C$ and $C'$.  Recall that $A$ and $B$ each have size at least $5$.  By symmetry between $A$ and $B$, we assume that we deleted at 
most $2$ vertices of $A$; so $A$ has at least $3$ remaining (undeleted) vertices.  Trivially, $B$ has at least $1$ remaining vertex.
So we can find an independent set $\{a_1'',a_2'',b_1''\}$ that is anticomplete to $\{i_3,i_4\}$. 
Now we move the tokens on $\{a_1',a_2',b_1\}$ to $\{a_1'',a_2'',b_1''\}$; the argument is similar to, but 
simpler than, the case above when we move $\{a_1,a_2\}$ to $\{a_1',a_2'\}$ and $\{b_1,b_2\}$ to 
$\{b_1',b_2'\}$.
We then perform the moves $j_3 \rightarrow i_3$, $a_1'' \rightarrow i_4$. 

To complete the transformation we distinguish $3$ cases.
If $y,y'\notin\{j_1,j_2,j_3\}$, then we complete the transformation with $a_2'' \rightarrow y$, $b_1''\rightarrow y'$, $i_4 \rightarrow x$, $i_3 \rightarrow j_3$.
If instead both $y,y' \in \{j_1,j_2,j_3\}$, then two of the positions $i_1,i_2,i_3,i_4,a_2'',b_1''$ are empty. Since this set is an independent set, we can move tokens to guarantee that $a_2''$ and $b_1''$ are empty; that is, there is no token in $A \cup B$. 
We move $i_3\to x, i_4\to j_3$.
Finally, if instead exactly one of $y,y'$ is in $\{j_1,j_2,j_3\}$, then we first suppose that this unique 
member is $y$. Now we assume that $a_2''$ is not in the current 
independent set. So we move: $b_1'' \rightarrow y'$, $i_3 \rightarrow x$, and $i_4 \rightarrow j_3$. If instead the unique member is $y'$, then $a_1''$ is not in the current independent set and we conclude with a similar argument.
\smallskip

\noindent
{\bf Case 2: $\bm{x \in \{j_1,j_2,j_3\}}$.} 
We again perform the first phase similarly except that if $x=j_1$, then we replace the move 
$j_1 \rightarrow b_1$ by $j_2 \rightarrow b_1$.
Applying Lemma~\ref{lem:deletion2perclass} to $A$, $B$, and $C$ ensures that there exist sets of three vertices in each of $A,B$ 
that are anticomplete to each other and both anticomplete to $C$.
So there exist $a_1',a_2'\in A$ and $b_1',b_2'\in B$ anticomplete to $C$.  Now we again shift the tokens in $A$ to $\{a_1',a_2'\}$
and shift the tokens in $B$ to $\{b_1',b_2'\}$.

If $x \in \{j_1,j_2\}$, then there is no token on $\{j_1,j_2\}$, so to complete the transformation we perform the moves $a_1' \rightarrow i_1$, $b_1' \rightarrow i_2$, $a_2' \rightarrow y$, and $b_2' \rightarrow y'$, and we are done. 
Thus we instead assume that $x=j_3$. Now we apply Lemma~\ref{lem:deletion2perclass} to $A$, $B$, and $C'$, instead of $C$. 
The above sequence transforms $I_s$ into $(I_s \setminus \{j_1,j_3\})\cup \{i_3,i_4\}$. 
To complete the transformation, we add the moves $j_2 \rightarrow i_1, i_3 \rightarrow i_2$, and $i_4 \rightarrow j_3$.
\end{proof}

\section*{Acknowledgments}
Thanks very much to 3 anonymous referees for their careful reading of our manuscript and their detailed feedback.
The subsequent revised version is improved because of it.

\bibliographystyle{habbrv}
{\footnotesize
%\small
\bibliography{token-jumping}
}

\end{document}